\documentclass[12pt,a4paper]{article}
\usepackage[utf8]{inputenc}
\usepackage{amsmath,amssymb,amsthm,amsfonts}
\usepackage{bm}
\usepackage{algorithmic}
\usepackage{mathtools}
\usepackage{algorithm}
\usepackage{geometry}
\usepackage{times}
\usepackage{hyperref}
\hypersetup{hidelinks}
\usepackage{cite}
\usepackage{xcolor}
\usepackage{enumitem}
\usepackage{bbm}
\usepackage{graphicx}
\usepackage{float}
\usepackage{booktabs}
\theoremstyle{plain}
\newtheorem{theorem}{Theorem}[section]
\newtheorem{lemma}[theorem]{Lemma}
\newtheorem{corollary}[theorem]{Corollary}
\newtheorem{proposition}[theorem]{Proposition}

\theoremstyle{definition}
\newtheorem{definition}[theorem]{Definition}
\newtheorem{example}[theorem]{Example}

\theoremstyle{remark}
\newtheorem{remark}[theorem]{Remark}
\newcommand{\Qs}{\mathbb{Q}_s}
\newcommand{\R}{\mathbb{R}}

\title{Iterative Methods for Computing the Moore--Penrose Inverse of Split-Quaternion Matrices with Applications}
\author{
Salman Ahmadi-Asl\thanks{Innopolis University, Innopolis, Russia, Email: s.ahmadiasl@innopolis.ru} \and
Valentin Leplat\thanks{ Innopolis University, Innopolis, Russia, Email: valentin.leplat@gmail.com} \and
Mohammad S. Alkousa\thanks{Innopolis University, Innopolis, Russia, Email: m.alkousa@innopolis.ru} 
}
\date{}
\begin{document}

\maketitle

\footnotetext[1]{The first and second authors contributed equally to this work.}

\begin{abstract}
We study iterative methods for computing the Moore--Penrose inverse of split-quaternion matrices. We first establish a consistent framework based on a \(2\times2\) real representation and the associated \(i\)-conjugate transpose. This representation gives a direct definition of the Moore--Penrose inverse and clarifies the treatment of nonzero zero divisors. We then analyze Newton--Schulz iterations for rectangular and rank-deficient matrices. Using a thin singular value decomposition of the real representative, we derive the convergence conditions, the evolution of the projector residuals, and an exact relation between the residual and the error. The resulting analysis applies both to an embedded real implementation and to a native split-quaternion iteration, which are shown to be equivalent under the representation. We also propose a low-degree polynomial initialization inspired by Souriau's inverse recursion. The polynomial is obtained by a least-squares approximation of the inverse Gram factor and is combined with a spectral acceptance test and a safe fallback initialization. Finally, we apply the proposed methods to cross and CUR approximations of split-quaternion matrices. We characterize the optimal middle factor for fixed sampled rows and columns and give conditions under which the cheaper cross factor gives an exact reconstruction. Numerical experiments illustrate the accuracy of the methods and the practical benefit of the polynomial initialization.

\end{abstract}

\section{Introduction}
\label{sec:introduction}
Quaternion matrix computations arise in least-squares problems, generalized
inverses, vector-sensor signal processing, and low-rank approximation
\cite{ahmadi2017iterative,ahmadi2017efficient,kyrchei2025quaternion,
le2004singular,ahmadi2026pass}. More broadly, quaternions are widely used
to represent rotations and orientations in computer graphics and robotics
\cite{shoemake1985animating,yuan1988closed}, and have also been studied in
quantum mechanics \cite{finkelstein1962foundations,de2000right}. Standard
quaternions, introduced by Hamilton in 1843, form a division algebra over
the real numbers, meaning that every nonzero quaternion has a multiplicative
inverse \cite{hamilton1843species}.

Split quaternions, introduced by Cockle \cite{cockle1849iii} in the
mid-nineteenth century, have a fundamentally different algebraic structure.
Unlike Hamilton quaternions, they contain nontrivial zero divisors and
nilpotent elements arising from their indefinite quadratic form. These
features make split quaternions suitable for certain physical models,
including electromagnetic field representations
\cite{ozdemir2022kinematic} and Lorentz transformations
\cite{antonuccio2015split,ozdemir2014eigenvalues}.

The matrix theory of Hamilton quaternions is now well developed, with
results on eigenvalue problems \cite{wood1985quaternionic,de2000right},
singular value decompositions \cite{le2004singular}, and generalized
inverses \cite{kyrchei2025quaternion}. The corresponding theory for split
quaternions has received comparatively less attention. At the scalar and
geometric levels, \"Ozdemir \cite{ozdemir2009roots} studied roots of split
quaternions and their geometric interpretation, while Kula and Yayl\i{}
\cite{kula2007split} considered split-quaternion representations in
Lorentzian geometry.

At the matrix level, split-quaternion matrices were studied
in~\cite{alagoz2012split}. More recently, matrix equations over split
quaternions have also received attention
\cite{si2024classical,liu2019consistency,gao2024anti}.

A fundamental difficulty in split-quaternion matrix theory is that the
natural multiplicative quadratic form is indefinite rather than positive
definite. Consequently, the usual spectral theory for real or complex
matrices does not transfer directly. The presence of zero divisors also
complicates the definition and computation of inverses and generalized
inverses. These features make the choice of adjoint and real representation
particularly important when defining the Moore--Penrose inverse.

A split quaternion is written as
\[
q=q_1+q_2\mathbf{i}+q_3\mathbf{j}+q_4\mathbf{k},
\qquad
q_1,q_2,q_3,q_4\in\mathbb{R},
\]
where
\[
\mathbf{i}^2=-1,
\qquad
\mathbf{j}^2=\mathbf{k}^2=1,
\qquad
\mathbf{i}\mathbf{j}=\mathbf{k}=-\mathbf{j}\mathbf{i},
\]
and consequently
\[
\mathbf{j}\mathbf{k}=-\mathbf{i}=-\mathbf{k}\mathbf{j},
\qquad
\mathbf{k}\mathbf{i}=\mathbf{j}=-\mathbf{i}\mathbf{k}.
\]

The Moore--Penrose inverse extends ordinary inversion to rectangular and
singular matrices. In the split-quaternion setting, its definition depends
on the adjoint used in the Penrose equations. The real representation is
therefore particularly useful: it makes this choice explicit and gives
access to standard real matrix algorithms.

The scalar Moore--Penrose inverse of split quaternions was studied
in~\cite{ablamowicz2020moore,cao2022moore}, while matrix extensions based
on real representations and split-quaternion decompositions were considered
in~\cite{wang2024singular}. Our recent work~\cite{leplat2025iterative}
developed iterative methods for computing the Moore--Penrose inverse of
Hamilton quaternion matrices. The split-quaternion setting requires
different algebraic conventions because of the presence of zero divisors
and the role of the \(i\)-conjugate transpose. Moreover, the thin-SVD
argument developed here removes the full-rank restriction used in the
direct Newton--Schulz analysis of~\cite{leplat2025iterative} and
establishes convergence for matrices of arbitrary rank.

Against this background, we analyze Newton--Schulz iterations for
split-quaternion matrices in two forms: an embedded implementation applied
to their real representatives and a native implementation written directly
in split-quaternion arithmetic. The two sequences are shown to be exactly
equivalent under the real representation. The main contributions are as
follows.

\begin{itemize}
    \item We state consistent conventions for the two conjugations, the Euclidean norm, the indefinite quadratic form, the \(i\)-conjugate transpose, and the real representation. We define the Moore--Penrose inverse through this representation and prove the scalar zero-divisor formula, existence, uniqueness, and the basic inherited identities.

    \item We give a complete Newton--Schulz analysis for rectangular and rank-deficient real matrices. A thin SVD yields exact formulas for the iterates and projector residuals, together with the error recurrence and an exact residual--error identity. The analysis is then transferred to both embedded and native split-quaternion iterations.

    \item We introduce a Souriau-inspired LS--Gram polynomial initialization. The polynomial is fitted to the actual Gram spectrum, checked against the Newton--Schulz convergence condition, and combined with a safe scaled-transpose fallback. We also compare its interpretation with a Chebyshev inverse polynomial.

    \item For cross and CUR approximations, we prove that \(C^\dagger A R^\dagger\) is a Frobenius-optimal middle factor for fixed sampled columns and rows. We also give a rank-revealing condition under which the cheaper factor \(W^\dagger\) coincides with it and gives exact reconstruction.

    \item MATLAB and Python implementations are used to assess the embedded iteration, the cross approximation, and the polynomial initialization.
\end{itemize}

The paper is organized as follows. Section~\ref{sec:preliminaries} introduces the algebraic conventions and the real representation. Section~\ref{Sec:MP} develops the Moore--Penrose inverse. Section~\ref{Sec:pro} presents the embedded and native Newton--Schulz methods and the polynomial initialization. Section~\ref{sec:cur} discusses cross and CUR approximations. Numerical experiments are reported in Section~\ref{sec:numerical}, and conclusions are given in Section~\ref{sec:conclusion}.

\section{Preliminaries}\label{sec:preliminaries}
The algebra of split quaternions is denoted by \(\mathbb{Q}_s\). It is isomorphic to the Clifford algebra \(C\ell_{1,1}\) and to the real matrix algebra \(\operatorname{Mat}(2,\mathbb{R})\). The set of \(m\times n\) split-quaternion matrices is denoted by \(\mathbb{Q}_s^{m\times n}\).

\subsection{Conjugations and basic scalar quantities}
For
\[
    q=q_1+q_2\mathbf{i}+q_3\mathbf{j}+q_4\mathbf{k},
\]
the full conjugate and the \(i\)-conjugate are, respectively,
\[
    \overline q=q_1-q_2\mathbf{i}-q_3\mathbf{j}-q_4\mathbf{k},
    \qquad
    \tau(q)=q_1-q_2\mathbf{i}+q_3\mathbf{j}+q_4\mathbf{k}.
\]
The full conjugation is used to define the split-quaternion quadratic form
\begin{equation}\label{norm_2}
    N(q):=q\overline q =q_1^2+q_2^2-q_3^2-q_4^2.
\end{equation}
The quantity \(N(q)\) is indefinite and is therefore not a norm. A nonzero scalar \(q\) is a zero divisor precisely when \(N(q)=0\). For numerical errors, we use the Euclidean component norm instead
\begin{equation}\label{norm}
    \|q\|_{\mathrm E}^2 :=q_1^2+q_2^2+q_3^2+q_4^2.
\end{equation}

\subsection{The real representation}
Define the scalar map \(\sigma:\mathbb{Q}_s\to\mathbb{R}^{2\times2}\) by
\begin{equation}\label{eq:sigma_scalar}
    q^\sigma :=
    \begin{bmatrix}
    q_1+q_3 & -q_2+q_4\\
    q_2+q_4 & q_1-q_3
    \end{bmatrix}.
\end{equation}
In particular,
\[
    1^\sigma=\begin{bmatrix}1&0\\0&1\end{bmatrix},\quad \mathbf{i}^\sigma=\begin{bmatrix}0&-1\\1&0\end{bmatrix},\quad \mathbf{j}^\sigma=\begin{bmatrix}1&0\\0&-1\end{bmatrix},\quad \mathbf{k}^\sigma=\begin{bmatrix}0&1\\1&0\end{bmatrix}.
\]
These four real matrices satisfy the same multiplication table as \(1,\mathbf{i},\mathbf{j},\mathbf{k}\). Hence, by bilinearity, the map \(\sigma\) preserves products of split-quaternion scalars.

For a split-quaternion matrix
\[
    A=A_1+A_2\mathbf{i}+A_3\mathbf{j}+A_4\mathbf{k} \in\mathbb{Q}_s^{m\times n}, \qquad A_\ell\in\mathbb{R}^{m\times n},
\]
we use the block extension considered in \cite{wang2024singular},
\begin{equation}\label{eq:sigma_matrix}
    A^\sigma := 
    \begin{bmatrix}
    A_1+A_3 & -A_2+A_4\\
    A_2+A_4 & A_1-A_3
    \end{bmatrix} \in\mathbb{R}^{2m\times2n}.
\end{equation}

\begin{lemma}[Properties of the real representation] \label{lem:sigma_properties}
Let \(A,B\) be split-quaternion matrices of compatible sizes and let \(\alpha\in\mathbb{R}\). Then
\begin{align}
(A+B)^\sigma&=A^\sigma+B^\sigma, \label{eq:sigma_add} \\ (\alpha A)^\sigma&=\alpha A^\sigma,\label{eq:sigma_scale} 
\\ (AB)^\sigma&=A^\sigma B^\sigma. \label{thir_pro}
\end{align}
The map \(\sigma\) is bijective. If
\[
Y=
\begin{bmatrix}
Y_{11}&Y_{12}\\
Y_{21}&Y_{22}
\end{bmatrix} \in\mathbb{R}^{2m\times2n},
    \qquad Y_{ij}\in\mathbb{R}^{m\times n},
\]
then \(Y=X^\sigma\) for the unique matrix
\[
    X=X_1+X_2\mathbf{i}+X_3\mathbf{j}+X_4\mathbf{k},
\]
where
\begin{equation}\label{eq:sigma_inverse}
\begin{aligned}
    X_1&=\frac{Y_{11}+Y_{22}}{2},
    &\qquad
    X_2&=\frac{Y_{21}-Y_{12}}{2},\\
    X_3&=\frac{Y_{11}-Y_{22}}{2},
    &
    X_4&=\frac{Y_{21}+Y_{12}}{2}.
\end{aligned}
\end{equation}
We denote this inverse map by \(\sigma^{-1}\).
\end{lemma}

\begin{proof}
The proof is given in Appendix~\ref{app:proof-sigma-properties}.
\end{proof}

The \(i\)-conjugation is the anti-involution associated with the real transpose.

\begin{lemma}[Transpose anti-involution] \label{lem:tau_transpose}
For all \(p,q\in\mathbb{Q}_s\),
\[
    \tau(\tau(q))=q,
    \qquad
    \tau(pq)=\tau(q)\tau(p),
    \qquad
    (\tau(q))^\sigma=(q^\sigma)^T.
\]
For a matrix
\[
    A=A_1+A_2\mathbf{i}+A_3\mathbf{j}+A_4\mathbf{k},
\]
define the \(i\)-conjugate transpose
\begin{equation}\label{eq:H_adjoint}
    A^H :=A_1^T-A_2^T\mathbf{i}+A_3^T\mathbf{j}+A_4^T\mathbf{k}.
\end{equation}
Then
\begin{equation}\label{eq:adjoint_representation}
    (A^H)^\sigma=(A^\sigma)^T, \qquad (AB)^H=B^H A^H.
\end{equation}
\end{lemma}

\begin{proof}
The proof is given in Appendix~\ref{app:proof-tau-transpose}.
\end{proof}

The full conjugate transpose
\[
    A^*:=A_1^T-A_2^T\mathbf{i}-A_3^T\mathbf{j}-A_4^T\mathbf{k}
\]
is a different operation. In general, \((A^*)^\sigma\neq(A^\sigma)^T\). The Moore--Penrose inverse considered in this paper is therefore defined with respect to \(H\), not with respect to \(*\).

We use the Euclidean Frobenius norm
\begin{equation}\label{eq:sq_frobenius}
    \|A\|_F^2 :=\sum_{i=1}^m\sum_{j=1}^n\|a_{ij}\|_{\mathrm E}^2 =\sum_{\ell=1}^4\|A_\ell\|_F^2.
\end{equation}

\begin{lemma}[Norm relation] \label{lem:norm_relation}
For every \(A\in\mathbb{Q}_s^{m\times n}\),
\begin{equation}\label{eq:norm_relation}
    \|A^\sigma\|_F^2=2\|A\|_F^2.
\end{equation}
\end{lemma}

\begin{proof}
The proof is given in Appendix~\ref{app:proof-norm-relation}.
\end{proof}

For quantities used in the convergence analysis, we define
\[
    \|A\|_{\sigma,2}:=\|A^\sigma\|_2, \qquad \operatorname{rank}_\sigma(A):=\operatorname{rank}(A^\sigma).
\]
Thus, all singular values and spectral condition numbers appearing below are those of the real representative. This convention avoids ambiguity between different notions of rank for split-quaternion matrices.

\section{The Moore--Penrose Pseudoinverse of Split Quaternions}\label{Sec:MP}
We first treat the scalar case and then define the matrix Moore--Penrose inverse through the real representation.

\subsection{The scalar Moore--Penrose inverse}
For a scalar \(q\in\mathbb{Q}_s\), its Moore--Penrose inverse with respect to the anti-involution \(\tau\) is the unique scalar \(q^\dagger\) satisfying
\begin{align}
    qq^\dagger q&=q,\label{eq:scalar_P1}\\
    q^\dagger q q^\dagger&=q^\dagger,\label{eq:scalar_P2}\\
    \tau(qq^\dagger)&=qq^\dagger,\label{eq:scalar_P3}\\
    \tau(q^\dagger q)&=q^\dagger q.\label{eq:scalar_P4}
\end{align}

\begin{proposition}[Scalar formula] \label{prop:scalar_mp}
Let \(q\in\mathbb{Q}_s\). Then
\begin{equation}\label{mp:sq}
    q^\dagger=
    \begin{cases}
    \displaystyle \frac{\overline q}{N(q)},
    & N(q)\neq0,\\[8pt]
    \displaystyle \frac{\tau(q)}{2\|q\|_{\mathrm E}^2},
    & N(q)=0,\ q\neq0,\\[8pt]
    0,&q=0.
    \end{cases}
\end{equation}
Moreover,
\[
    (q^\dagger)^\sigma=(q^\sigma)^\dagger.
\]
\end{proposition}

\begin{proof}
Let \(Q=q^\sigma\). Direct calculation gives
\begin{equation}\label{eq:det_sigma_q}
    \det(Q) =(q_1+q_3)(q_1-q_3)-(-q_2+q_4)(q_2+q_4) =N(q).
\end{equation}

Assume first that \(N(q)\neq0\). Then \(Q\) is invertible and its Moore--Penrose inverse is its ordinary inverse. Since
\[
    (\overline q)^\sigma =
    \begin{bmatrix}
    q_1-q_3&q_2-q_4\\
    -q_2-q_4&q_1+q_3
    \end{bmatrix},
\]
the usual \(2\times2\) inverse formula and \eqref{eq:det_sigma_q} give
\[
    Q^{-1}=\frac{(\overline q)^\sigma}{N(q)}.
\]
By injectivity of \(\sigma\), this corresponds to
\[
    q^{-1}=\frac{\overline q}{N(q)}.
\]
Since an invertible element has Moore--Penrose inverse equal to its ordinary inverse, the first branch follows.

Assume next that \(N(q)=0\) and \(q\neq0\). Then \(Q\neq0\), \(\det(Q)=0\), and therefore \(Q\) has rank one. For a nonzero rank-one real matrix \(Q=uv^T\),
\[
    Q^\dagger =\frac{vu^T}{\|u\|_2^2\|v\|_2^2} =\frac{Q^T}{\|Q\|_F^2}.
\]
Indeed, substituting this expression verifies the four real Penrose equations. In the present case, Lemma~\ref{lem:tau_transpose} gives
\[
    Q^T=(\tau(q))^\sigma,
\]
while direct expansion yields
\[
    \|Q\|_F^2 =(q_1+q_3)^2+(-q_2+q_4)^2+(q_2+q_4)^2+(q_1-q_3)^2 =2\|q\|_{\mathrm E}^2.
\]
Consequently,
\[
    Q^\dagger =\frac{(\tau(q))^\sigma}{2\|q\|_{\mathrm E}^2} =\left(\frac{\tau(q)}{2\|q\|_{\mathrm E}^2}\right)^\sigma.
\]
This proves the second branch. The case \(q=0\) is immediate. Since the real Moore--Penrose inverse is unique and \(\sigma\) is bijective, the scalar satisfying \eqref{eq:scalar_P1}--\eqref{eq:scalar_P4} is unique.
\end{proof}

\begin{example}[A nonzero zero divisor]
Let \(q=1+\mathbf{j}\). Then \(N(q)=0\) and
\[
    q^\sigma=\begin{bmatrix}2&0\\0&0\end{bmatrix}, \qquad (q^\sigma)^\dagger=\begin{bmatrix}1/2&0\\0&0\end{bmatrix}.
\]
Formula \eqref{mp:sq} gives
\[
    q^\dagger=\frac{1+\mathbf{j}}{4},
\]
whose real representative is exactly \((q^\sigma)^\dagger\).
\end{example}

\subsection{The matrix Moore--Penrose inverse}
The real representation provides a direct and unambiguous matrix definition.

\begin{definition}[Moore--Penrose inverse]
\label{def:matrix_mp}
Let \(A\in\mathbb{Q}_s^{m\times n}\). Its Moore--Penrose inverse with respect to the adjoint \(H\) is
\begin{equation} \label{eq:matrix_mp_definition}
    A^\dagger :=\sigma^{-1}\!\left((A^\sigma)^\dagger\right) \in\mathbb{Q}_s^{n\times m}.
\end{equation}
\end{definition}

\begin{theorem}[Existence, uniqueness, and the Penrose equations] \label{lem_mp_em}
For every \(A\in\mathbb{Q}_s^{m\times n}\), the matrix \(A^\dagger\) in Definition~\ref{def:matrix_mp} is the unique matrix \(X\in\mathbb{Q}_s^{n\times m}\) satisfying
\begin{align}
    AXA&=A,\label{eq:matrix_P1}\\
    XAX&=X,\label{eq:matrix_P2}\\
    (AX)^H&=AX,\label{eq:matrix_P3}\\
    (XA)^H&=XA.\label{eq:matrix_P4}
\end{align}
Moreover,
\begin{equation} \label{eq:mp_commutes_sigma}
    (A^\dagger)^\sigma=(A^\sigma)^\dagger.
\end{equation}
\end{theorem}

\begin{proof}
Set \(M=A^\sigma\) and \(Y=M^\dagger\). By definition, \(X=\sigma^{-1}(Y)\), so \(X^\sigma=Y\). Applying \(\sigma\) to the left-hand side of \eqref{eq:matrix_P1} and using Lemma~\ref{lem:sigma_properties},
\[
    (AXA)^\sigma =A^\sigma X^\sigma A^\sigma =MYM =M =A^\sigma.
\]
The injectivity of \(\sigma\) implies \(AXA=A\). Similarly,
\[
    (XAX)^\sigma=YMY=Y=X^\sigma,
\]
which gives \eqref{eq:matrix_P2}.

For the third equation, Lemma~\ref{lem:tau_transpose} yields
\[
    ((AX)^H)^\sigma =((AX)^\sigma)^T =(MY)^T.
\]
Since \(Y=M^\dagger\), the real Penrose equations give \((MY)^T=MY=(AX)^\sigma\). Injectivity then gives \((AX)^H=AX\). The proof of \eqref{eq:matrix_P4} is identical, using \((YM)^T=YM\).

It remains to prove uniqueness without assuming it in advance. Let \(Z\in\mathbb{Q}_s^{n\times m}\) satisfy \eqref{eq:matrix_P1}--\eqref{eq:matrix_P4}. Applying \(\sigma\) shows that \(Z^\sigma\) satisfies the four real Penrose equations for \(M=A^\sigma\). The real Moore--Penrose inverse is unique, hence
\[
    Z^\sigma=M^\dagger=Y=X^\sigma.
\]
Since \(\sigma\) is injective, \(Z=X\). This proves uniqueness and also establishes \eqref{eq:mp_commutes_sigma}.
\end{proof}

\begin{proposition}[Basic properties] \label{prop:mp_basic_properties}
Let \(A\in\mathbb{Q}_s^{m\times n}\), let \(\lambda\in\mathbb{R}\setminus\{0\}\), and let \(U,V\) be square split-quaternion matrices satisfying
\[
    U^HU=UU^H=I, \qquad V^HV=VV^H=I.
\]
Then
\begin{align*}
    (A^\dagger)^\dagger&=A,\\
    (A^H)^\dagger&=(A^\dagger)^H,\\
    (\lambda A)^\dagger&=\lambda^{-1}A^\dagger,\\
    (UAV)^\dagger&=V^H A^\dagger U^H.
\end{align*}
Moreover, \(AA^\dagger\) and \(A^\dagger A\) are \(H\)-Hermitian idempotents.
\end{proposition}

\begin{proof}
The proof is given in Appendix~\ref{app:proof-basic-mp}.
\end{proof}

\begin{theorem}[Split quaternion SVD] \cite{wang2024singular}
Let \( A \in \Qs^{m \times n} \). There exist unitary split quaternion matrices \( U \in \Qs^{m \times m} \), \( V \in \Qs^{n \times n} \) and a diagonal matrix \( \Sigma \in \Qs^{m \times n} \) with diagonal entries of the form \( \sigma_i = a_i + b_i\mathbf{j} \) (with \( a_i, b_i \in \R \), \( a_i \geq 0 \)) such that:
\[
    A = U \Sigma V^H.
\]
\end{theorem}

\begin{remark}[Relation with the split-quaternion SVD]
Several properties derived above can also be obtained from the split-quaternion singular value decomposition developed in \cite{wang2024singular}. Under the corresponding definitions of split-quaternion unitary matrices and diagonal factors, a matrix \(A \in \mathbb{Q}_s^{m \times n}\) can be written as
\[
    A = U \Sigma V^H,
\]
and its Moore--Penrose inverse can be expressed as
\[
    A^\dagger = V \Sigma^\dagger U^H.
\]

In this paper, we instead define and analyze the Moore--Penrose inverse directly through the real representation:
\[
    A^\dagger = \sigma^{-1}\left((A^\sigma)^\dagger\right).
\]
This approach is particularly convenient for our purposes. First, it makes the role of the \(i\)-conjugate transpose explicit through the identity
\[
    (A^H)^\sigma = (A^\sigma)^T.
\]
Second, it provides a direct connection with the real Newton--Schulz iteration studied in the next section.

The split-quaternion SVD therefore gives an alternative characterization of the Moore--Penrose inverse, while the real representation provides the main foundation for the analysis and algorithms developed in this paper.
\end{remark}

\begin{corollary}[Transfer of convergent real sequences]
\label{cor_2}
Let \(A\in\mathbb{Q}_s^{m\times n}\), and let
\[
    X_k\in\mathbb{R}^{2n\times2m}
\]
satisfy \(X_k\to(A^\sigma)^\dagger\) in the real Frobenius norm. Then
\[
    \widehat X_k:=\sigma^{-1}(X_k)
    \longrightarrow A^\dagger
    \quad\text{in }\mathbb{Q}_s^{n\times m},
\]
and
\begin{equation} \label{eq:transfer_norm}
    \|\widehat X_k-A^\dagger\|_F =\frac{1}{\sqrt2}\|X_k-(A^\sigma)^\dagger\|_F.
\end{equation}
\end{corollary}

\begin{proof}
By linearity of \(\sigma^{-1}\),
\[
    \widehat X_k-A^\dagger =\sigma^{-1}\!\left(X_k-(A^\sigma)^\dagger\right).
\]
The norm identity follows from Lemma~\ref{lem:norm_relation}. Since the right-hand side tends to zero, so does the left-hand side.
\end{proof}

The definition suggests the following direct method.

\begin{algorithm}[H]
\caption{Direct computation of \(A^\dagger\) through the real representation.}
\label{MP_Alg}
\begin{algorithmic}[1]
\STATE \textbf{Input:} \(A=A_1+A_2\mathbf{i}+A_3\mathbf{j}+A_4\mathbf{k}\in\mathbb{Q}_s^{m\times n}\).
\STATE \textbf{Output:} \(A^\dagger\in\mathbb{Q}_s^{n\times m}\).
\STATE Form \(M=A^\sigma\in\mathbb{R}^{2m\times2n}\) using \eqref{eq:sigma_matrix}.
\STATE Compute \(Y=M^\dagger\in\mathbb{R}^{2n\times2m}\), for example by a real SVD.
\STATE Partition \(Y=\begin{bmatrix}Y_{11}&Y_{12}\\Y_{21}&Y_{22}\end{bmatrix}\) into \(n\times m\) blocks.
\STATE Recover \(A^\dagger=\sigma^{-1}(Y)\) using \eqref{eq:sigma_inverse}.
\RETURN \(A^\dagger\).
\end{algorithmic}
\end{algorithm}

The direct method is robust and provides a useful reference solution. For large dense representatives, however, an explicit SVD may be costly. This motivates iterative methods based primarily on matrix--matrix multiplication. Such methods do not reduce the cubic dense asymptotic order automatically, but they can be attractive when optimized multiplication kernels, suitable matrix structure, or moderate-accuracy approximations are available.

\section{Proposed Iterative Algorithms}
\label{Sec:pro}
We first analyze Newton--Schulz \cite{schulz1933iterative,ben2003generalized} for an arbitrary real rectangular matrix. The split-quaternion results will then follow through the real representation.

\subsection{Newton--Schulz for a real rectangular matrix}

For a square nonsingular matrix \(M\in\mathbb{R}^{p\times p}\), the Newton--Schulz iteration for approximating \(M^{-1}\) is
\begin{equation} \label{eq:NS-square}
    X_{k+1} = X_k(2I_p-MX_k) = 2X_k-X_kMX_k.
\end{equation}
It can be obtained by applying Newton's method to the nonlinear matrix
equation
\[
    X^{-1}-M=0.
\]
Indeed, the Fr\'echet derivative of
\[
    F(X)=X^{-1}-M
\]
at \(X\) is
\[
    DF(X)[H]=-X^{-1}HX^{-1}.
\]
The Newton correction \(H_k\) therefore satisfies
\[
    -X_k^{-1}H_kX_k^{-1} = -\bigl(X_k^{-1}-M\bigr),
\]
which gives
\[
    H_k=X_k-X_kMX_k.
\]
Hence
\[
    X_{k+1}=X_k+H_k=2X_k-X_kMX_k.
\]

The main property of this iteration is the squaring of the inverse
residual. If
\[
    R_k=I_p-MX_k,
\]
then
\[
    R_{k+1} = I_p-MX_{k+1} = R_k^2.
\]
Thus, if \(\|R_0\|_2<1\), the residual converges quadratically to zero.

For a rectangular matrix \(M\in\mathbb{R}^{p\times q}\), there is no two-sided inverse. In the full-row-rank case,
\[
    M^\dagger=M^T(MM^T)^{-1}, \qquad MM^\dagger=I_p,
\]
whereas in the full-column-rank case,
\[
    M^\dagger=(M^TM)^{-1}M^T, \qquad M^\dagger M=I_q.
\]
This motivates the rectangular Newton--Schulz iteration
\begin{equation} \label{eq:NS_real}
    X_{k+1} = X_k(2I_p-MX_k) = (2I_q-X_kM)X_k = 2X_k-X_kMX_k, \qquad X_0=\alpha M^T,
\end{equation}
where \(X_k\in\mathbb{R}^{q\times p}\). The two forms in \eqref{eq:NS_real} are algebraically equivalent. In an implementation, one may use the form involving the smaller square intermediate matrix.

When \(M\) is rank deficient, neither \(MM^\dagger\) nor \(M^\dagger M\) is an identity matrix. They are the orthogonal projectors
\[
    P_L=MM^\dagger, \qquad P_R=M^\dagger M.
\]
This does not prevent convergence from the initialization \(X_0=\alpha M^T\). As shown in the next theorem, the iterates remain in the nonzero singular subspaces of \(M\), every positive singular direction converges independently to the corresponding coefficient of \(M^\dagger\), and
\[
    MX_k\longrightarrow P_L, \qquad X_kM\longrightarrow P_R.
\]

\begin{theorem}[Exact Newton--Schulz convergence] \label{theorem1}
Let \(M\in\mathbb{R}^{p\times q}\) be nonzero and have rank \(r\), with positive singular values
\[
    \sigma_1\geq\sigma_2\geq\cdots\geq\sigma_r>0.
\]
Let \(\{X_k\}\) be generated by \eqref{eq:NS_real}. If
\begin{equation} \label{eq:alpha_interval}
    0<\alpha<\frac{2}{\sigma_1^2},
\end{equation}
then \(X_k\to M^\dagger\). More precisely, if
\[
    \theta_i:=1-\alpha\sigma_i^2, \qquad i=1,2,\ldots,r,
\]
then
\begin{equation}\label{eq:Xk_exact_formula}
    X_k =V_r\operatorname{diag}\!\left( \frac{1-\theta_1^{2^k}}{\sigma_1},\frac{1-\theta_2^{2^k}}{\sigma_2},\ldots, \frac{1-\theta_r^{2^k}}{\sigma_r} \right)U_r^T,
\end{equation}
where \(M=U_r\Sigma_rV_r^T\) is a thin SVD. In particular, with
\[
    P_L:=MM^\dagger=U_rU_r^T, \qquad P_R:=M^\dagger M=V_rV_r^T,
\]
one has
\begin{align}
    P_L-MX_k &=U_r\operatorname{diag}(\theta_1^{2^k},\theta_2^{2^k},\ldots,\theta_r^{2^k})U_r^T, \label{eq:left_projector_residual}
    \\ P_R-X_kM &=V_r\operatorname{diag}(\theta_1^{2^k},\theta_2^{2^k},\ldots,\theta_r^{2^k})V_r^T. \label{eq:right_projector_residual}
\end{align}
Hence, if
\[
    \rho_0:=\max_{1\leq i\leq r}|1-\alpha\sigma_i^2|<1,
\]
then
\begin{equation} \label{eq:projector_residual_rate}
    \|P_L-MX_k\|_2 =\|P_R-X_kM\|_2 =\rho_0^{2^k}.
\end{equation}
For the zero matrix, \(M^\dagger=0\) and the initialization gives \(X_k=0\) for all \(k\).
\end{theorem}

\begin{proof}
Let
\[
    M=U_r\Sigma_rV_r^T, \qquad \Sigma_r=\operatorname{diag}(\sigma_1,\sigma_2,\ldots,\sigma_r),
\]
be a thin SVD. Since
\[
    X_0=\alpha M^T =V_r(\alpha\Sigma_r)U_r^T,
\]
we first show by induction that every iterate has the form
\begin{equation} \label{eq:Xk_invariant_form}
    X_k=V_rD_kU_r^T,
\end{equation}
where \(D_k\) is diagonal. The statement is true at \(k=0\). If it holds at iteration \(k\), then
\begin{align*}
    X_{k+1}
    & =2V_rD_kU_r^T -V_rD_kU_r^T U_r\Sigma_rV_r^T V_rD_kU_r^T
    \\& =V_r\bigl(2D_k-D_k\Sigma_rD_k\bigr)U_r^T.
\end{align*}
Thus \eqref{eq:Xk_invariant_form} is preserved, with
\[
    D_{k+1}=2D_k-D_k\Sigma_rD_k.
\]
Writing \(D_k=\operatorname{diag}(d_{1,k},\ldots,d_{r,k})\), each singular direction evolves independently according to
\begin{equation} \label{eq:scalar_NS}
    d_{i,k+1}=d_{i,k}(2-\sigma_i d_{i,k}), \qquad d_{i,0}=\alpha\sigma_i.
\end{equation}
Define
\[
    e_{i,k}:=1-\sigma_i d_{i,k}.
\]
Then \eqref{eq:scalar_NS} gives
\begin{align*}
    e_{i,k+1}
    &=1-\sigma_i d_{i,k}(2-\sigma_i d_{i,k})\\
    &=1-2\sigma_i d_{i,k}+\sigma_i^2d_{i,k}^2\\
    &=(1-\sigma_i d_{i,k})^2 =e_{i,k}^2.
\end{align*}
Since \(e_{i,0}=1-\alpha\sigma_i^2=\theta_i\), repeated squaring yields
\[
    e_{i,k}=\theta_i^{2^k}.
\]
Solving for \(d_{i,k}\) gives
\[
    d_{i,k}=\frac{1-\theta_i^{2^k}}{\sigma_i},
\]
which proves \eqref{eq:Xk_exact_formula}.

Condition \eqref{eq:alpha_interval} implies
\[
    -1<1-\alpha\sigma_1^2 \leq 1-\alpha\sigma_i^2<1
\]
for each \(i\), so \(|\theta_i|<1\). Therefore \(\theta_i^{2^k}\to0\), and \eqref{eq:Xk_exact_formula} gives
\[
    X_k\longrightarrow V_r\Sigma_r^{-1}U_r^T =M^\dagger.
\]
Finally,
\begin{align*}
    MX_k &=U_r\Sigma_rD_kU_r^T =U_r\operatorname{diag}(1-\theta_i^{2^k})U_r^T,
    \\ X_kM &=V_rD_k\Sigma_rV_r^T =V_r\operatorname{diag}(1-\theta_i^{2^k})V_r^T,
\end{align*}
which proves \eqref{eq:left_projector_residual}--\eqref{eq:projector_residual_rate}.
\end{proof}

\begin{remark}[Optional regularization]
Theorem~\ref{theorem1} shows that the Newton--Schulz iteration converges directly to \(M^\dagger\) for matrices of arbitrary rank when \(X_0=\alpha M^T\). Regularization is therefore not required for this convergence result.

It can nevertheless be useful for noisy or nearly rank-deficient problems. For \(\lambda>0\), one may consider
\[
    X_\lambda = (M^TM+\lambda I_q)^{-1}M^T = M^T(MM^T+\lambda I_p)^{-1}.
\]
Using a thin singular value decomposition, one obtains
\[
    X_\lambda\longrightarrow M^\dagger \qquad\text{as}\qquad \lambda\downarrow0.
\]
This regularized approach was studied for quaternion matrices in \cite{leplat2025iterative}. Here, the thin-SVD analysis gives the stronger direct convergence result, while regularization remains a useful filtering strategy when small singular values should be damped.
\end{remark}

The thin-SVD proof also gives the exact error recurrence and explains why the four Penrose conditions hold at the limit.

\begin{corollary}[Error recurrence and quadratic bound] \label{cor:NS_error_recurrence}
Under the assumptions of Theorem~\ref{theorem1}, let
\[
    E_k:=X_k-M^\dagger.
\]
Then
\begin{equation} \label{eq:error_recurrence}
    E_{k+1}=-E_kME_k,
\end{equation}
and therefore
\begin{equation} \label{eq:quadratic_F_bound}
    \|E_{k+1}\|_F \leq \|M\|_2\|E_k\|_F^2.
\end{equation}
Moreover, \(MX_k\) and \(X_kM\) are symmetric for every \(k\), and the limit \(M^\dagger\) satisfies all four real Penrose equations.
\end{corollary}

\begin{proof}
From \eqref{eq:Xk_invariant_form},
\[
    E_k =V_r(D_k-\Sigma_r^{-1})U_r^T.
\]
Using the scalar recurrence,
\begin{align*}
    d_{i,k+1}-\frac1{\sigma_i}
    &=2d_{i,k}-\sigma_i d_{i,k}^2-\frac1{\sigma_i}\\
    &=-\sigma_i\left(d_{i,k}-\frac1{\sigma_i}\right)^2.
\end{align*}
Collecting the singular directions gives \eqref{eq:error_recurrence}. The Frobenius bound follows from
\[
    \|E_kME_k\|_F \leq\|E_k\|_F\|M\|_2\|E_k\|_2 \leq\|M\|_2\|E_k\|_F^2.
\]
The matrices \(MX_k\) and \(X_kM\) are symmetric by the explicit formulas in the proof of Theorem~\ref{theorem1}. Passing to the limit gives the two symmetry Penrose equations. The remaining equations follow either from the limit \(M^\dagger\), or directly from the thin-SVD expression
\[
    MM^\dagger M=M, \qquad M^\dagger MM^\dagger=M^\dagger.
\]
\end{proof}

\begin{theorem}[Exact relation between the residual and the error]
\label{thm:residual_error}
Under the assumptions of Theorem~\ref{theorem1}, define
\[
    R_k:=MX_kM-M.
\]
Then
\begin{equation} \label{eq:residual_error_exact}
    E_k=M^\dagger R_kM^\dagger.
\end{equation}
Consequently,
\begin{equation} \label{eq:residual_error_bound}
    \|X_k-M^\dagger\|_F \leq \|M^\dagger\|_2^2 \cdot \|MX_kM-M\|_F.
\end{equation}
Equivalently,
\begin{equation} \label{eq:relative_residual_error_bound}
    \frac{\|X_k-M^\dagger\|_F}{\|M^\dagger\|_F} \leq \frac{\|M^\dagger\|_2^2\|M\|_F}{\|M^\dagger\|_F} \frac{\|MX_kM-M\|_F}{\|M\|_F}.
\end{equation}
\end{theorem}

\begin{proof}
The first Penrose equation gives
\[
    R_k=M(X_k-M^\dagger)M=ME_kM.
\]
The thin-SVD form of \(E_k\) shows that
\[
    E_k=(M^\dagger M)E_k(MM^\dagger).
\]
Indeed, \(M^\dagger M=V_rV_r^T\), \(MM^\dagger=U_rU_r^T\), and \(E_k=V_r(D_k-\Sigma_r^{-1})U_r^T\). Therefore,
\begin{align*}
    M^\dagger R_kM^\dagger = M^\dagger ME_kMM^\dagger =(M^\dagger M)E_k(MM^\dagger)=E_k,
\end{align*}
which proves \eqref{eq:residual_error_exact}. Taking Frobenius norms gives \eqref{eq:residual_error_bound}, and multiplying and dividing by \(\|M\|_F\) and \(\|M^\dagger\|_F\) gives \eqref{eq:relative_residual_error_bound}.
\end{proof}

Thus, for the exact Newton--Schulz sequence initialized by \(\alpha M^T\), the normalized residual
\[
    \frac{\|MX_kM-M\|_F}{\|M\|_F}
\]
is a valid stopping quantity. In finite precision, it is also useful to report all four normalized Penrose residuals,
\begin{align*}
    r_1(X)&:=\frac{\|MXM-M\|_F}{\|M\|_F},
    & r_2(X)&:=\frac{\|XMX-X\|_F}{\|X\|_F},\\ r_3(X)&:=\frac{\|(MX)^T-MX\|_F}{\|MX\|_F}, 
    & r_4(X)&:=\frac{\|(XM)^T-XM\|_F}{\|XM\|_F},
\end{align*}
with the usual convention that a ratio is set to zero when both numerator and denominator vanish. These quantities reveal loss of the exact invariant structure caused by roundoff.

\subsection{Choice of the initial scaling}
The admissible interval \eqref{eq:alpha_interval} is the basic convergence requirement. Different optimization criteria lead to different distinguished choices of \(\alpha\).

\begin{proposition}[Scaling criteria] \label{prop:alpha_choices}
Let \(M\neq0\) have rank \(r\) and positive singular values \(\sigma_1\geq\cdots\geq\sigma_r>0\).
\begin{enumerate}
    \item If \(\rho\geq\|M\|_2^2=\sigma_1^2\) and \(0<\beta<2\), then $\displaystyle \alpha=\beta / \rho$ satisfies \(0<\alpha<2/\sigma_1^2\). A simple conservative choice is \(\beta=1\).

    \item The unconstrained minimizer of
    \[
        \|\alpha M^T-M^\dagger\|_F
    \]
    over \(\alpha\in\mathbb{R}\) is
    \begin{equation} \label{eq:alpha_F}
        \alpha_F=\frac{r}{\sum_{i=1}^r\sigma_i^2} =\frac{r}{\|M\|_F^2}.
    \end{equation}
    This value is an approximation-optimal scalar, but it need not satisfy the convergence interval \eqref{eq:alpha_interval}.

    \item The minimizer over \(\alpha>0\) of the worst initial projector residual
    \[
        \max_{1\leq i\leq r}|1-\alpha\sigma_i^2|
    \]
    is
    \begin{equation} \label{eq:alpha_minimax}
        \alpha_{\mathrm{mm}} =\frac{2}{\sigma_1^2+\sigma_r^2}.
    \end{equation}
    It satisfies the convergence interval and yields
    \[
        \max_i|1-\alpha_{\mathrm{mm}}\sigma_i^2| =\frac{\sigma_1^2-\sigma_r^2}{\sigma_1^2+\sigma_r^2}.
    \]
\end{enumerate}
\end{proposition}

\begin{proof}
The first statement follows from
\[
    0<\frac{\beta}{\rho} \leq\frac{\beta}{\sigma_1^2} <\frac{2}{\sigma_1^2}.
\]

For the second statement, use a thin SVD:
\begin{align*}
    \|\alpha M^T - M^\dagger\|_F^2 =\left\|V_r (\alpha\Sigma_r - Sigma_r^{-1})U_r^T\right\|_F^2 =\sum_{i=1}^r\left(\alpha\sigma_i-\frac1{\sigma_i}\right)^2.
\end{align*}
Differentiating with respect to \(\alpha\) gives
\[
    2\sum_{i=1}^r\sigma_i \left(\alpha\sigma_i-\frac1{\sigma_i}\right) =2\left(\alpha\sum_{i=1}^r\sigma_i^2-r\right).
\]
The second derivative is \(2\sum_i\sigma_i^2>0\), so the unique minimizer is \eqref{eq:alpha_F}.

For the third statement, set \(a=\sigma_r^2\) and \(b=\sigma_1^2\). Since \(t\mapsto1-\alpha t\) is affine, the maximum of \(|1-\alpha t|\) for \(t\in[a,b]\) is attained at an endpoint. Hence the objective is
\[
    \max\{|1-\alpha a|,|1-\alpha b|\}.
\]
At the minimizer, the two endpoint errors have equal magnitude and opposite signs; otherwise \(\alpha\) can be moved slightly to decrease the larger one. Thus
\[
    1-\alpha a=-(1-\alpha b),
\]
which gives \(\alpha=2/(a+b)\). Substituting this value gives the stated residual factor. Since \(a>0\),
\[
    0<\frac{2}{a+b}<\frac{2}{b}=\frac{2}{\sigma_1^2},
\]
so the convergence condition holds.
\end{proof}

For a rectangular matrix, \(\|M\|_2^2\) should be estimated through the largest eigenvalue of \(M^TM\) or \(MM^T\), not through powers of \(M\). Power iteration or Lanczos applied to the smaller of these two positive semidefinite matrices gives a practical estimate. A safety factor should be used if the estimate is not a certified upper bound.

\subsection{Embedded split-quaternion iteration}
We now apply the real iteration to \(M=A^\sigma\).

\begin{algorithm}[H]
\caption{Embedded Newton--Schulz method for a split-quaternion Moore--Penrose inverse.}
\label{alg:newton_split}
\begin{algorithmic}[1]
\STATE \textbf{Input:} \(A\in\mathbb{Q}_s^{m\times n}\), tolerance \(\varepsilon>0\), maximum iterations \(K_{\max}\), and \(\beta\in(0,2)\).
\STATE \textbf{Output:} \(\widehat X\approx A^\dagger\in\mathbb{Q}_s^{n\times m}\).
\STATE Form \(M=A^\sigma\in\mathbb{R}^{2m\times2n}\).
\IF{\(M=0\)}
    \RETURN \(0\).
\ENDIF
\STATE Choose a conservative \(\rho\geq\|M\|_2^2\), and set \(\alpha=\beta/\rho\).
\STATE Set \(X_0=\alpha M^T\).
\FOR{\(k=0,1,\ldots,K_{\max}-1\)}
    \IF{\(m\leq n\)}
        \STATE \(X_{k+1}=X_k(2I_{2m}-MX_k)\).
    \ELSE
        \STATE \(X_{k+1}=(2I_{2n}-X_kM)X_k\).
    \ENDIF
    \STATE \(r_{k+1}=\|MX_{k+1}M-M\|_F/\|M\|_F\).
    \IF{\(r_{k+1}\leq\varepsilon\)}
        \STATE \textbf{break}.
    \ENDIF
\ENDFOR
\STATE Set \(\widehat X=\sigma^{-1}(X_{k+1})\).
\RETURN \(\widehat X\).
\end{algorithmic}
\end{algorithm}

\begin{theorem}[Convergence in split-quaternion coordinates]
\label{thm:embedded_split_convergence}
Let \(A\in\mathbb{Q}_s^{m\times n}\), set \(M=A^\sigma\), and suppose \(M\neq0\). Let \(X_k\) be generated by the embedded iteration with
\[
    0<\alpha<\frac{2}{\|M\|_2^2}, \qquad \widehat X_k:=\sigma^{-1}(X_k).
\]
Then \(\widehat X_k\to A^\dagger\). If
\[
    \widehat E_k:=\widehat X_k-A^\dagger,
\]
then
\begin{equation} \label{eq:split_quadratic_bound}
    \|\widehat E_{k+1}\|_F \leq\sqrt2\,\|A\|_{\sigma,2}\|\widehat E_k\|_F^2.
\end{equation}
\end{theorem}

\begin{proof}
Theorem~\ref{theorem1} gives \(X_k\to M^\dagger\), while Theorem~\ref{lem_mp_em} gives \(M^\dagger=(A^\dagger)^\sigma\). Corollary~\ref{cor_2} therefore implies \(\widehat X_k\to A^\dagger\).

For the quantitative bound, let \(E_k=X_k-M^\dagger\). By Lemma~\ref{lem:norm_relation},
\[
    \|E_k\|_F=\sqrt2\|\widehat E_k\|_F.
\]
Using \eqref{eq:quadratic_F_bound},
\begin{align*}
    \|\widehat E_{k+1}\|_F
    &=\frac1{\sqrt2}\|E_{k+1}\|_F\\
    &\leq\frac1{\sqrt2}\|M\|_2\|E_k\|_F^2\\
    &=\sqrt2\,\|M\|_2\|\widehat E_k\|_F^2.
\end{align*}
Since \(\|M\|_2=\|A\|_{\sigma,2}\), this proves \eqref{eq:split_quadratic_bound}.
\end{proof}

\subsection{Native split-quaternion iteration}
The native iteration is
\begin{equation} \label{eq:NS_native}
    \widehat X_{k+1} =\widehat X_k(2I_m-A\widehat X_k), \qquad \widehat X_0=\alpha A^H,
\end{equation}
where \(\widehat X_k\in\mathbb{Q}_s^{n\times m}\). It is not necessary to develop a separate inverse-based convergence argument. The native sequence is exactly the real Newton--Schulz sequence transported through \(\sigma^{-1}\).

\begin{theorem}[Equivalence and convergence of the native iteration] \label{thm:native_equivalence}
Let \(A\in\mathbb{Q}_s^{m\times n}\), let \(M=A^\sigma\), and assume
\[
    0<\alpha<\frac{2}{\|M\|_2^2}.
\]
Let \(\widehat X_k\) be generated by \eqref{eq:NS_native}, and define \(X_k=(\widehat X_k)^\sigma\). Then
\[
    X_0=\alpha M^T, \qquad X_{k+1}=X_k(2I_{2m}-MX_k).
\]
Consequently,
\[
    \widehat X_k\longrightarrow A^\dagger.
\]
Furthermore, with \(\widehat E_k=\widehat X_k-A^\dagger\),
\begin{equation}\label{eq:native_error_recurrence}
    \widehat E_{k+1}=-\widehat E_kA\widehat E_k,
\end{equation}
and the bound \eqref{eq:split_quadratic_bound} holds.
\end{theorem}

\begin{proof}
By Lemma~\ref{lem:tau_transpose},
\[
    X_0=(\alpha A^H)^\sigma =\alpha(A^H)^\sigma =\alpha M^T.
\]
Assume \(X_k=(\widehat X_k)^\sigma\). Applying \(\sigma\) to \eqref{eq:NS_native} and using the product property,
\begin{align*}
    X_{k+1}
    &=(\widehat X_{k+1})^\sigma\\
    &=(\widehat X_k)^\sigma
    \left(2I_{2m}-A^\sigma(\widehat X_k)^\sigma\right)\\
    &=X_k(2I_{2m}-MX_k).
\end{align*}
Thus, by induction, the represented native iterates coincide exactly with the real Newton--Schulz iterates. Theorem~\ref{theorem1} and Corollary~\ref{cor_2} give convergence to \(A^\dagger\).

The real error satisfies \(E_{k+1}=-E_kME_k\). Since
\[
    E_k=(\widehat E_k)^\sigma, \qquad M=A^\sigma,
\]
we obtain
\[
    (\widehat E_{k+1})^\sigma = (-\widehat E_kA\widehat E_k)^\sigma.
\]
Injectivity of \(\sigma\) proves \eqref{eq:native_error_recurrence}. The norm bound is the one established in Theorem~\ref{thm:embedded_split_convergence}.
\end{proof}

\begin{algorithm}[H]
\caption{Native Newton--Schulz method for a split-quaternion Moore--Penrose inverse.}
\label{alg:native_newton}
\begin{algorithmic}[1]
\STATE \textbf{Input:} \(A\in\mathbb{Q}_s^{m\times n}\), tolerance \(\varepsilon>0\), maximum iterations \(K_{\max}\), and \(\beta\in(0,2)\).
\STATE \textbf{Output:} \(\widehat X\approx A^\dagger\in\mathbb{Q}_s^{n\times m}\).
\IF{\(A=0\)}
    \RETURN \(0\).
\ENDIF
\STATE Choose a conservative \(\rho\geq\|A\|_{\sigma,2}^2\), and set \(\alpha=\beta/\rho\).
\STATE Set \(\widehat X_0=\alpha A^H\).
\FOR{\(k=0,1,\ldots,K_{\max}-1\)}
    \IF{\(m\leq n\)}
        \STATE \(\widehat X_{k+1}=\widehat X_k(2I_m-A\widehat X_k)\).
    \ELSE
        \STATE \(\widehat X_{k+1}=(2I_n-\widehat X_kA)\widehat X_k\).
    \ENDIF
    \STATE \(r_{k+1}=\|A\widehat X_{k+1}A-A\|_F/\|A\|_F\).
    \IF{\(r_{k+1}\leq\varepsilon\)}
        \STATE \textbf{break}.
    \ENDIF
\ENDFOR
\RETURN \(\widehat X_{k+1}\).
\end{algorithmic}
\end{algorithm}

The \(\sigma\)-representation and a native four-component storage both contain \(4mn\) real scalar entries. Thus, there is no intrinsic raw-storage factor between them. A native implementation may nevertheless avoid explicit block assembly and can exploit componentwise structure, while a real implementation can use highly optimized real matrix multiplication. Both approaches have the same dense asymptotic order; the faster implementation depends on dimensions, arithmetic kernels, data layout, and hardware.

\subsection{A residual-checked extrapolation heuristic}
A possible postprocessing candidate is
\begin{equation} \label{eq:extrapolation_heuristic}
    \widetilde X_{k+1} =X_{k+1}+\omega(X_{k+1}-X_k) =(1+\omega)X_{k+1}-\omega X_k.
\end{equation}
This construction should be regarded as a heuristic, not as a quadratically convergent Newton--Schulz method. Indeed, if \(E_k=X_k-M^\dagger\), then exactly
\begin{equation} \label{eq:extrapolation_error_exact}
    \widetilde X_{k+1}-M^\dagger =(1+\omega)E_{k+1}-\omega E_k.
\end{equation}
Since \(E_{k+1}=\mathcal{O}(\|E_k\|^2)\), any fixed \(\omega\neq0\) gives
\[
    \widetilde X_{k+1}-M^\dagger = -\omega E_k+\mathcal{O}(\|E_k\|^2).
\]
Therefore, a fixed nonzero extrapolation parameter does not preserve quadratic convergence and may worsen the approximation close to the limit. If \eqref{eq:extrapolation_heuristic} is tested numerically, a safe interpretation is to use it only as a postprocessed candidate and accept it only when its true residual is smaller than that of \(X_{k+1}\). Feeding the extrapolated point back into the iteration defines a different two-step method and requires a separate stability analysis.

\subsection{A Souriau-inspired polynomial initialization for Newton--Schulz} \label{subsec:souriau-poly-init}

The convergence of Newton--Schulz depends strongly on the initial point.  A scaled transpose is simple and safe, but it applies the same scalar correction to all singular directions.  We now construct a low-degree polynomial initialization that partially corrects the spectrum before the Newton--Schulz refinement starts.  The idea is motivated by Souriau's finite recursion for the inverse of a square matrix~\cite{Souriau1948, Souriau2007}.

\paragraph{Motivation from Souriau's recursion.}
Let $B\in\mathbb{R}^{N\times N}$.  Define
\[
    P_1=I_N, \qquad k_1=-\operatorname{tr}(B),
\]
and, for $r=2,\ldots,N$,
\[
    P_r=P_{r-1}B+k_{r-1}I_N, \qquad k_r=-\frac{1}{r}\operatorname{tr}(P_rB).
\]
Souriau's recursion \cite{Souriau1948,Souriau2007} gives
\begin{equation} \label{eq:souriau-identity}
    P_NB+k_NI_N=0.
\end{equation}
Indeed, the induction gives
\[
    P_r=B^{r-1}+k_1B^{r-2}+\cdots+k_{r-1}I_N.
\]
The trace formula for $k_r$ is then Newton's identity for the coefficients of
\[
    \chi_B(t)=t^N+k_1t^{N-1}+\cdots+k_N.
\]
Equation \eqref{eq:souriau-identity} follows from the Cayley--Hamilton theorem.  Since $k_N=(-1)^N\det(B)$, a nonsingular matrix satisfies
\begin{equation} \label{eq:souriau-inverse}
    B^{-1}=-\frac{1}{k_N}P_N.
\end{equation}
Thus, an exact inverse can be written as a polynomial in the matrix.

We do not apply this high-degree formula directly to a rectangular matrix.  We keep a simpler idea: a low-degree polynomial can approximate the inverse of a square Gram factor, while Newton--Schulz performs the final refinement.

\paragraph{The LS--Gram polynomial.}
Let $M=A^\sigma\in\mathbb{R}^{p\times q}$
be the real representation of a split-quaternion matrix.  Assume first that $p\ge q$ and that $M$ has full column rank.  Then
\[
    M^\dagger=(M^TM)^{-1}M^T.
\]
Choose
\[
    \rho\geq\|M\|_2^2
\]
and define
\begin{equation}\label{eq:normalized-column-gram}
    B=\frac{M^TM}{\rho}.
\end{equation}
The matrix $B$ is symmetric positive definite and its spectrum is contained in $(0,1]$.  Since
\[
    M^\dagger=\frac{1}{\rho}B^{-1}M^T,
\]
we approximate $B^{-1}$ by
\[
    p_d(t)=\sum_{\ell=0}^{d}c_\ell t^\ell
\]
and set
\begin{equation}\label{eq:LS-Gram-column-init}
    X_0^{(d)} = \frac{1}{\rho}p_d(B)M^T.
\end{equation}
The coefficients are selected from
\begin{equation}\label{eq:LS-Gram-problem}
    \min_{c_0,\ldots,c_d} \left\|I_q-Bp_d(B)\right\|_F^2.
\end{equation}
This is a linear least-squares problem.  More precisely, if
\[
    \mathcal{K}_d = 
    \begin{bmatrix}
    \operatorname{vec}(B)&
    \operatorname{vec}(B^2)&
    \cdots&
    \operatorname{vec}(B^{d+1})
    \end{bmatrix},
\]
then \eqref{eq:LS-Gram-problem} is equivalent to
\begin{equation}\label{eq:vectorized-LS-Gram}
    \min_{c\in\mathbb{R}^{d+1}} \left\| \operatorname{vec}(I_q)-\mathcal{K}_dc \right\|_2^2.
\end{equation}
A QR or SVD least-squares solver can therefore be used.  If the coefficient vector is not unique, the fitted matrix $Bp_d(B)$ is still unique, because it is the orthogonal projection of $I_q$ onto
\[
    \operatorname{span}\{B,B^2,\ldots,B^{d+1}\}.
\]
Since $B$ is nonsingular, the fitted matrices $p_d(B)$ and $X_0^{(d)}$ are also unique.

The spectral meaning of the fit is simple.

\begin{proposition}[Spectral form of the LS--Gram objective]\label{prop:LS-Gram-spectral}
Let $\lambda_1,\lambda_2,\ldots,\lambda_q$ be the eigenvalues of $B$, counted with multiplicity.  Then
\begin{equation} \label{eq:spectral-LS-Gram}
    \left\|I_q-Bp_d(B)\right\|_F^2 = \sum_{i=1}^{q} \left(1-\lambda_i p_d(\lambda_i)\right)^2.
\end{equation}
For $d=0$, the minimizer is
\begin{equation}\label{eq:degree-zero-LS-scaling}
    c_0 = \frac{\operatorname{tr}(B)}{\operatorname{tr}(B^2)}, \qquad X_0^{(0)}=\frac{c_0}{\rho}M^T.
\end{equation}
\end{proposition}

\begin{proof}
Let $B=V\Lambda V^T$ be an eigendecomposition.  Since
\[
    I_q-Bp_d(B) = V\bigl(I_q-\Lambda p_d(\Lambda)\bigr)V^T,
\]
orthogonal invariance of the Frobenius norm gives \eqref{eq:spectral-LS-Gram}.  When $d=0$, the objective becomes
\[
    \|I_q-c_0B\|_F^2 =q-2c_0\operatorname{tr}(B)+c_0^2\operatorname{tr}(B^2).
\]
Differentiating with respect to $c_0$ gives \eqref{eq:degree-zero-LS-scaling}.
\end{proof}

Thus the LS polynomial is adapted to the actual eigenvalues of the normalized Gram matrix.  Degree zero already gives an optimized scalar multiple of $M^T$.  It is not necessarily the same as the conservative scaling in Proposition~\ref{prop:alpha_choices}.

\paragraph{Convergence of the polynomially initialized iteration.}
The Frobenius objective in \eqref{eq:LS-Gram-problem} is useful for fitting the polynomial, but Newton--Schulz convergence is controlled by the largest spectral residual.

\begin{theorem}[Polynomially initialized Newton--Schulz]
\label{thm:polynomial-NS}
Let $M\in\mathbb{R}^{p\times q}$ be nonzero.

\begin{enumerate}[label=\textup{(\roman*)}]
\item Suppose that $M$ has full column rank.  Let $B$ and $X_0^{(d)}$ be defined by \eqref{eq:normalized-column-gram} and \eqref{eq:LS-Gram-column-init}, and set
\begin{equation} \label{eq:column-polynomial-error}
    \eta_d = \max_{\lambda\in\sigma(B)} \left|1-\lambda p_d(\lambda)\right|.
\end{equation}

\item Suppose that $M$ has full row rank.  Define
\begin{equation} \label{eq:normalized-row-gram}
    C=\frac{MM^T}{\rho}, \qquad X_0^{(d)} = M^T\frac{1}{\rho}p_d(C),
\end{equation}
and set
\begin{equation} \label{eq:row-polynomial-error}
    \eta_d = \max_{\lambda\in\sigma(C)} \left|1-\lambda p_d(\lambda)\right|.
\end{equation}
\end{enumerate}

Starting from $X_0^{(d)}$, consider
\begin{equation}\label{eq:polynomial-NS-iteration}
    X_{k+1}=2X_k-X_kMX_k.
\end{equation}
If $\eta_d<1$, then $X_k\to M^\dagger$.  In the full-column-rank case,
\begin{equation} \label{eq:column-polynomial-residual-rate}
    I_q-X_kM = \left(I_q-X_0^{(d)}M\right)^{2^k}, \qquad \|I_q-X_kM\|_2=\eta_d^{2^k}.
\end{equation}
In the full-row-rank case,
\begin{equation}
\label{eq:row-polynomial-residual-rate}
I_p-MX_k
=
\left(I_p-MX_0^{(d)}\right)^{2^k},
\qquad
\|I_p-MX_k\|_2=\eta_d^{2^k}.
\end{equation}
\end{theorem}

\begin{proof}
We give the full-column-rank proof.  The row case is identical after exchanging $M^TM$ and $MM^T$.  Let
\[
    M=U\Sigma V^T, \qquad B=V\Lambda V^T, \qquad \lambda_i=\frac{\sigma_i^2}{\rho}.
\]
From \eqref{eq:LS-Gram-column-init},
\[
    X_0^{(d)} = V\operatorname{diag}\!\left( \frac{\lambda_i p_d(\lambda_i)}{\sigma_i} \right)U^T.
\]
The proof of Theorem~\ref{theorem1} shows that Newton--Schulz preserves this singular-vector form.  If $d_{i,k}$ denotes its coefficient in the $i$th singular direction and
\[
    e_{i,k}=1-\sigma_i d_{i,k},
\]
then
\[
    e_{i,k+1}=e_{i,k}^2, \qquad  e_{i,0}=1-\lambda_i p_d(\lambda_i).
\]
Hence
\[
    e_{i,k}=e_{i,0}^{2^k}.
\]
The condition $\eta_d<1$ gives $e_{i,k}\to0$ for every $i$, so $d_{i,k}\to\sigma_i^{-1}$ and $X_k\to M^\dagger$.  Moreover,
\[
    I_q-X_kM = V\operatorname{diag}(e_{i,k})V^T,
\]
which gives \eqref{eq:column-polynomial-residual-rate}.  The same argument with the left residual gives \eqref{eq:row-polynomial-residual-rate}.
\end{proof}

If \(0<\eta_d<1\), a residual below
\(\varepsilon\in(0,1)\) is obtained once
\begin{equation}\label{eq:polynomial-iteration-count}
    k \geq \max\left\{0, \left\lceil \log_2\!\left( \frac{\log(\varepsilon)}{\log(\eta_d)} \right) \right\rceil \right\}.
\end{equation}
The outer maximum accounts for the case \(\eta_d\leq\varepsilon\), in which the initial point already satisfies the prescribed tolerance. This also explains why a moderate improvement of the initial spectral residual can remove several Newton--Schulz steps.

\begin{remark}[Acceptance test and safe fallback]
\label{rem:LS-does-not-guarantee-NS}
Minimizing the Frobenius residual does not automatically give $\eta_d<1$.  For example, let
\[
B=\operatorname{diag}(1,\underbrace{10^{-2},\ldots,10^{-2}}_{1000\text{ times}}).
\]
For $d=0$, formula \eqref{eq:degree-zero-LS-scaling} gives $c_0=10$, and the residual at the eigenvalue $1$ is $1-c_0=-9$.  The degree-zero LS fit therefore lies outside the Newton--Schulz convergence region.

After fitting $p_d$, we consequently estimate
\[
\eta_d=\|I-Bp_d(B)\|_2
\]
(or its row analogue).  We accept the polynomial initialization only when $\eta_d<1-\delta$, where $\delta\in(0,1)$ is a small safety margin.  Otherwise, we use
\begin{equation}
\label{eq:LS-Gram-safe-fallback}
X_0^{\mathrm{safe}}=\frac{\beta}{\rho}M^T,
\qquad 0<\beta<2.
\end{equation}
Theorem~\ref{theorem1} shows that this fallback converges for matrices of arbitrary rank.  We also use it when full row or column rank has not been certified.
\end{remark}

\paragraph{Compatibility with the split-quaternion structure.}
The construction stays in the image of the real representation in exact arithmetic. Indeed,
\[
    M^TM=(A^HA)^\sigma, \qquad MM^T=(AA^H)^\sigma.
\]
Therefore, in the full-column-rank case,
\begin{equation} \label{eq:column-structure-poly}
    \frac{1}{\rho}p_d\!\left(\frac{M^TM}{\rho}\right)M^T = \left[ \frac{1}{\rho} p_d\!\left(\frac{A^HA}{\rho}\right)A^H \right]^\sigma,
\end{equation}
and an analogous identity holds in the full-row-rank case.  Moreover, if $X_k=Z_k^\sigma$, then
\[
    2X_k-X_kMX_k=(2Z_k-Z_kAZ_k)^\sigma.
\]
Thus, under the chosen representation, the polynomial initialization and all Newton--Schulz iterates correspond exactly to matrices over
\(\mathbb{Q}_s\).

\begin{remark}[Relation with Chebyshev inverse polynomials]
\label{rem:LS-versus-Chebyshev}
The LS--Gram polynomial minimizes a discrete error over the eigenvalues of $B$, as shown in \eqref{eq:spectral-LS-Gram}. A Chebyshev polynomial follows a different approach: it controls the largest error over a prescribed interval containing the spectrum.

Assume that
\[
    \sigma(B)\subseteq [a,1], \qquad 0<a<1.
\]
The goal is to construct a polynomial $p_d$ that approximates $t^{-1}$ on $[a,1]$. Equivalently, we want the residual
\[
    r_{d+1}(t)=1-tp_d(t)
\]
to be small on this interval.

The affine map
\[
    t\longmapsto \frac{1+a-2t}{1-a}
\]
maps $[a,1]$ onto $[-1,1]$. Since the Chebyshev polynomial $T_{d+1}$ of the first kind satisfies
\[
    \left|T_{d+1}(s)\right|\leq 1, \qquad s \in[-1,1],
\]
we define
\[
    r_{d+1}(t) = \frac{T_{d+1}\!\left(\frac{1+a-2t}{1-a}\right)}{T_{d+1}\!\left(\frac{1+a}{1-a}\right)}, \qquad p_d^{\mathrm{Ch}}(t) = \frac{1-r_{d+1}(t)}{t}.
\]
The normalization ensures that
\[
    r_{d+1}(0)=1.
\]
Hence $1-r_{d+1}(t)$ vanishes at $t=0$ and is divisible by $t$. Therefore, $p_d^{\mathrm{Ch}}$ is a polynomial of degree at most $d$.

Let
\[
    \vartheta = \frac{1-\sqrt{a}}{1+\sqrt{a}}.
\]
Then
\begin{equation} \label{eq:Chebyshev-inverse-bound}
    \max_{t\in[a,1]} \left|1-tp_d^{\mathrm{Ch}}(t)\right| =\max_{t\in[a,1]} \left|r_{d+1}(t)\right| \leq \frac{2\vartheta^{d+1}} {1+\vartheta^{2(d+1)}} \leq 2\vartheta^{d+1}.
\end{equation}
Thus, the approximation error decreases geometrically with the polynomial degree.

In the Gram setting, if
\[
    B=\frac{M^TM}{\rho}, \qquad \rho=\|M\|_2^2,
\]
then
\[
    \sigma(B) \subseteq \left[\kappa_2(M)^{-2},1\right].
\]

Taking
\[
    a=\kappa_2(M)^{-2}
\]
gives
\[
    \vartheta = \frac{\kappa_2(M)-1} {\kappa_2(M)+1}.
\]

The Chebyshev construction gives an explicit worst-case guarantee over the full spectral interval, but it requires a reliable lower bound on the spectrum. The LS--Gram polynomial instead adapts to the observed eigenvalues, but it does not automatically provide the same uniform guarantee and must pass the acceptance test described above. Neither construction is uniformly better. We use the LS--Gram polynomial here, while the Chebyshev polynomial provides a natural comparison and a possible initialization strategy.
\end{remark}

\paragraph{Cost and implementation.}
Let $s=q$ in the full-column-rank case and $s=p$ in the full-row-rank case.  The initialization forms an $s\times s$ Gram matrix, computes the powers $B, B^2, \ldots, B^{d+1}$, solves a least-squares problem with $d+1$ unknowns, and checks the spectral residual.  These costs must be included in runtime comparisons.  The normalization by $\rho$ controls the largest eigenvalue. In the full-column-rank case,
\[
    \kappa_2(M^TM)=\kappa_2(M)^2,
\]
and a high-degree monomial basis can be ill-conditioned.  We therefore use low degrees and solve \eqref{eq:vectorized-LS-Gram} by QR or SVD rather than by explicit normal equations.

\begin{algorithm}[H]
\caption{LS--Gram initialized Newton--Schulz}
\label{alg:souriau-poly-ns}
\begin{algorithmic}[1]
\REQUIRE Split-quaternion matrix $A$, degree $d$, tolerance $\varepsilon$, safety margin $\delta\in(0,1)$, and $\beta\in(0,2)$.
\STATE Form $M=A^\sigma$ and choose $\rho\geq\|M\|_2^2$.
\IF{$M=0$}
    \STATE \textbf{return} $0$.
\ENDIF
\STATE Set the safe initial point $X_0=(\beta/\rho)M^T$.
\IF{$M$ is certified to have full column rank}
    \STATE Set $B=M^TM/\rho$ and solve \eqref{eq:LS-Gram-problem}.
    \STATE Set $X_{\mathrm{poly}}=\rho^{-1}p_d(B)M^T$ and estimate $\eta=\|I-X_{\mathrm{poly}}M\|_2$.
    \IF{$\eta<1-\delta$}
        \STATE Set $X_0=X_{\mathrm{poly}}$.
    \ENDIF
\ELSIF{$M$ is certified to have full row rank}
    \STATE Set $C=MM^T/\rho$ and solve the analogous least-squares problem.
    \STATE Set $X_{\mathrm{poly}}=M^T\rho^{-1}p_d(C)$ and estimate $\eta=\|I-MX_{\mathrm{poly}}\|_2$.
    \IF{$\eta<1-\delta$}
        \STATE Set $X_0=X_{\mathrm{poly}}$.
    \ENDIF
\ENDIF
\STATE Apply $X_{k+1}=2X_k-X_kMX_k$ until
\[
    \frac{\|MX_kM-M\|_F}{\|M\|_F}\leq\varepsilon.
\]
\STATE \textbf{return} \(\sigma^{-1}(X_k)\).
\end{algorithmic}
\end{algorithm}

The polynomial does not replace Newton--Schulz, and it is not an exact Souriau inversion of the rectangular matrix.  It gives a matrix-dependent approximation of the inverse Gram factor.  When its accepted residual $\eta_d$ is smaller than the residual of the scaled transpose, Theorem~\ref{thm:polynomial-NS} predicts fewer Newton--Schulz iterations.  The numerical experiments will determine whether this reduction compensates for the additional initialization cost.

\section{Application to cross and CUR approximations}
\label{sec:cur}

Cross and CUR approximations represent a matrix using a small number of its
columns and rows.  Let
\[
    A\in\mathbb{Q}_s^{m\times n},
\]
and choose index sets
\[
    \mathcal{C}\subset\{1,2,\ldots,n\}, \qquad \mathcal{R}\subset\{1,2,\ldots,m\}, \qquad |\mathcal{C}|=|\mathcal{R}|=k.
\]
We define
\[
    C=A(:,\mathcal{C})\in\mathbb{Q}_s^{m\times k}, \qquad R=A(\mathcal{R},:)\in\mathbb{Q}_s^{k\times n},
\]
and the intersection matrix
\[
    W=A(\mathcal{R},\mathcal{C})\in\mathbb{Q}_s^{k\times k}.
\]
A CUR approximation has the form
\[
    A\approx CUR, \qquad U\in\mathbb{Q}_s^{k\times k}.
\]

Throughout this section, rank, range, and orthogonality are understood through the real representation introduced in Section~\ref{sec:preliminaries}.  In particular,
\[
    \operatorname{rank}_{\sigma}(A) = \operatorname{rank}(A^\sigma).
\]
We also use the represented column and row spaces
\[
    \mathcal{C}_{\sigma}(A) = \operatorname{range}(A^\sigma), \qquad \mathcal{R}_{\sigma}(A) =  \operatorname{range}\bigl((A^\sigma)^T\bigr).
\]

\subsection{The optimal middle matrix for fixed columns and rows}

For fixed \(C\) and \(R\), the middle matrix can be selected by solving
\[
    \min_{U\in\mathbb{Q}_s^{k\times k}} \|A-CUR\|_F.
\]
The next result gives a minimizer without requiring \(C\) to have full column rank or \(R\) to have full row rank.

\begin{theorem}[Optimal middle matrix for fixed \(C\) and \(R\)]
\label{thm:cur-optimal-middle}
Let
\[
    A\in\mathbb{Q}_s^{m\times n}, \qquad C\in\mathbb{Q}_s^{m\times k}, \qquad R\in\mathbb{Q}_s^{k\times n}.
\]
Define
\[
    U_\star=C^\dagger A R^\dagger.
\]
Then \(U_\star\) minimizes \(\|A-CUR\|_F\) over \(U\in\mathbb{Q}_s^{k\times k}\).  More precisely, if
\[
    P_C=CC^\dagger, \qquad P_R=R^\dagger R,
\]
then
\[
    CU_\star R=P_C A P_R,
\]
and, for every \(U\in\mathbb{Q}_s^{k\times k}\),
\begin{equation} \label{eq:cur-pythagorean}
    \|A-CUR\|_F^2 = \|A-P_C A P_R\|_F^2 + \|P_C A P_R-CUR\|_F^2.
\end{equation}
Consequently,
\[
    \min_U\|A-CUR\|_F = \|A-P_C A P_R\|_F.
\]
The minimizing middle matrix need not be unique.  It is unique if \(C^\sigma\) has full column rank and \(R^\sigma\) has full row rank.
\end{theorem}

\begin{proof}
Set
\[
    M=A^\sigma,\qquad F=C^\sigma,\qquad G=R^\sigma.
\]
By Theorem~\ref{lem_mp_em},
\[
    (U_\star)^\sigma = F^\dagger M G^\dagger.
\]
Moreover,
\[
    (P_C)^\sigma=FF^\dagger=:P, \qquad  (P_R)^\sigma=G^\dagger G=:Q.
\]
The matrices \(P\) and \(Q\) are the real orthogonal projectors onto \(\operatorname{range}(F)\) and \(\operatorname{range}(G^T)\), respectively. Hence
\[
    (CU_\star R)^\sigma = FF^\dagger M G^\dagger G = PMQ.
\]

We next show that \(PMQ\) is the orthogonal projection of \(M\), in the Frobenius inner product, onto the linear space
\[
    \mathcal{S} = \{FZG:Z\in\mathbb{R}^{2k\times 2k}\}.
\]
Indeed, every \(Y=FZG\in\mathcal{S}\) satisfies
\[
    PYQ=Y.
\]
Using \(P^T=P\), \(Q^T=Q\), and \(P^2=P\), \(Q^2=Q\), we obtain
\[
\begin{aligned}
    \langle M-PMQ,Y \rangle_F
    &= \operatorname{tr}\bigl((M-PMQ)^T Y\bigr)\\
    &=\operatorname{tr}(M^T Y) - \operatorname{tr}(Q M^T P Y)\\
    &= \operatorname{tr}(M^T Y) - \operatorname{tr}(M^T P Y Q)\\
    &=0.
\end{aligned}
\]
Thus \(M-PMQ\) is orthogonal to every matrix in \(\mathcal{S}\).  Since \(F U^\sigma G\in\mathcal{S}\), the Pythagorean identity gives
\[
    \|M-FU^\sigma G\|_F^2 = \|M-PMQ\|_F^2 + \|PMQ-FU^\sigma G\|_F^2.
\]
Finally, the norm relation
\[
    \|Z^\sigma\|_F^2=2\|Z\|_F^2
\]
transfers this identity to split-quaternion matrices and yields
\eqref{eq:cur-pythagorean}.

If \(C^\sigma\) has full column rank and \(R^\sigma\) has full row rank, the map \(Z\mapsto C^\sigma ZR^\sigma\) is injective. Hence, the minimizing middle matrix is unique.
\end{proof}

\begin{corollary}[Exact CUR representation]\label{cor:cur-exact-ranges}
Under the assumptions of Theorem~\ref{thm:cur-optimal-middle}, suppose that
\[
    \mathcal{C}_{\sigma}(A) \subseteq \mathcal{C}_{\sigma}(C), \qquad \mathcal{R}_{\sigma}(A) \subseteq \mathcal{R}_{\sigma}(R).
\]
Then
\[
    A=C U_\star R, \qquad U_\star=C^\dagger A R^\dagger.
\]
\end{corollary}

\begin{proof}
The first inclusion implies
\[
    (CC^\dagger A)^\sigma = C^\sigma(C^\sigma)^\dagger A^\sigma = A^\sigma,
\]
and therefore \(CC^\dagger A=A\). The second inclusion similarly gives \(AR^\dagger R=A\). Hence
\[
    CU_\star R = CC^\dagger A R^\dagger R = A.
\]
\end{proof}

When \(C\) and \(R\) are formed from columns and rows of \(A\), their represented column and row spaces are already contained in those of \(A\). Therefore, the assumptions of Corollary~\ref{cor:cur-exact-ranges} are equivalent to
\[
    \operatorname{rank}_{\sigma}(C) = \operatorname{rank}_{\sigma}(R) = \operatorname{rank}_{\sigma}(A).
\]
No assumption of the form \(\operatorname{rank}_{\sigma}(C)=k\) is needed.  In particular, the number \(k\) of selected split-quaternion rows and columns should not be confused with the rank of the \(2m\times 2n\) real representative.

\subsection{The cross middle matrix}

The optimal factor \(U_\star=C^\dagger A R^\dagger\) involves the full matrix \(A\).  A cheaper cross approximation uses only the intersection matrix:
\[
    U_{\mathrm{cross}}=W^\dagger, \qquad A\approx C W^\dagger R.
\]
In general, \(W^\dagger\) is not equal to \(U_\star\), and the cross approximation is not the Frobenius-optimal approximation for fixed \(C\) and \(R\).  The two factors coincide under the exact rank-revealing condition given below.

\begin{theorem}[Exact cross approximation]
\label{thm:exact-cross}
Let \(C\), \(R\), and \(W\) be obtained from \(A\) as above.  If
\[
    \operatorname{rank}_{\sigma}(W) = \operatorname{rank}_{\sigma}(A),
\]
then
\[
    W^\dagger=C^\dagger A R^\dagger
\]
and
\[
    A=CW^\dagger R.
\]
\end{theorem}

\begin{proof}
Let
\[
    M=A^\sigma, \qquad F=C^\sigma, \qquad G=R^\sigma, \qquad H=W^\sigma.
\]
The selected split-quaternion rows and columns correspond to paired real rows and columns in the representation.  Thus, for suitable real selection matrices \(S\) and \(T\),
\[
    F=MT,\qquad G=S^T M,\qquad H=S^TMT.
\]

Let \(r=\operatorname{rank}(M)\) and choose a rank factorization
\[
    M=LK,
\]
where \(L\in\mathbb{R}^{2m\times r}\) has full column rank and \(K\in\mathbb{R}^{r\times 2n}\) has full row rank.  Define
\[
    K_{\mathcal C}=KT, \qquad L_{\mathcal R}=S^TL.
\]
Then
\[
    F=LK_{\mathcal C}, \qquad G=L_{\mathcal R}K, \qquad H=L_{\mathcal R}K_{\mathcal C}.
\]
The assumption \(\operatorname{rank}(H)=r\) implies that \(L_{\mathcal R}\) has full column rank and \(K_{\mathcal C}\) has full row rank.

For a product \(XY\) with \(X\) full column rank and \(Y\) full row rank, the reverse-order identity
\[
    (XY)^\dagger=Y^\dagger X^\dagger
\]
follows directly from
\[
    X^\dagger=(X^TX)^{-1}X^T, \qquad Y^\dagger=Y^T(YY^T)^{-1}.
\]
Applying this identity gives
\[
    H^\dagger = K_{\mathcal C}^\dagger L_{\mathcal R}^\dagger,
\]
as well as
\[
    F^\dagger=K_{\mathcal C}^\dagger L^\dagger, \qquad G^\dagger=K^\dagger L_{\mathcal R}^\dagger.
\]
Consequently,
\[
\begin{aligned}
    F^\dagger M G^\dagger
    &= K_{\mathcal C}^\dagger L^\dagger L K K^\dagger L_{\mathcal R}^\dagger\\
    &= K_{\mathcal C}^\dagger L_{\mathcal R}^\dagger\\
    &= H^\dagger.
\end{aligned}
\]
By compatibility of the Moore--Penrose inverse with \(\sigma\),
\[
    (W^\dagger)^\sigma = H^\dagger = F^\dagger M G^\dagger = (C^\dagger A R^\dagger)^\sigma.
\]
Injectivity of \(\sigma\) yields
\[
    W^\dagger=C^\dagger A R^\dagger.
\]
Finally, Theorem~\ref{thm:cur-optimal-middle} and the rank condition imply that the represented column and row spaces of \(C\) and \(R\) coincide with those of \(A\).  Corollary~\ref{cor:cur-exact-ranges} therefore gives
\[
    A=CW^\dagger R.
\]
\end{proof}

\begin{remark}
If \(W^\sigma\) is nonsingular, then \(W^\dagger=W^{-1}\).  For a singular or rectangular-rank cross, the Moore--Penrose inverse remains well defined. However, when
\[
    \operatorname{rank}_{\sigma}(W) < \operatorname{rank}_{\sigma}(A),
\]
the factor \(W^\dagger\) should be viewed as a practical cross choice, not as the optimal middle matrix from Theorem~\ref{thm:cur-optimal-middle}.
\end{remark}

\subsection{Practical construction and cost}

The two middle factors have different purposes.
\begin{itemize}
\item The factor
\[
    U_\star=C^\dagger A R^\dagger
\]
is optimal for the fixed sampled matrices \(C\) and \(R\), but its formation uses the full matrix \(A\).  For dense matrices, the multiplication involving \(A\) has a cost proportional to \(mnk\), in addition to the pseudoinverses of \(C\) and \(R\).

\item The factor
\[
    U_{\mathrm{cross}}=W^\dagger
\]
requires only the pseudoinverse of the \(k\times k\) split-quaternion core. Through the real representation, this is the pseudoinverse of a \(2k\times2k\) real matrix and has cubic cost in \(k\).  The approximation may be stored in factored form as \((C,W^\dagger,R)\), with storage proportional to
\[
    (m+n)k+k^2.
\]
Forming the complete matrix \(CW^\dagger R\) has an additional cost proportional to \(mnk\) and is unnecessary when only matrix-vector products or selected entries are required.
\end{itemize}

The index sets may be selected by pivoted QR, leverage-score sampling, max-volume procedures~\cite{goreinov1997pseudo,allen2024maximal,kliushev2026quaternionmaximumvolumesubmatrixselection}, or another rank-revealing method.  If the selection is performed on \(A^\sigma\), the real rows and columns must be selected in the paired blocks induced by the split-quaternion representation.  An arbitrary selection of individual real rows or columns does not necessarily correspond to a split-quaternion submatrix.

The construction used in the numerical section is summarized in the following.

\begin{algorithm}[H]
\caption{Cross and CUR approximation of a split-quaternion matrix}
\label{alg:split-cur}
\begin{algorithmic}[1]
\REQUIRE
\(A\in\mathbb{Q}_s^{m\times n}\), index sets \(\mathcal C\subset\{1,2,\ldots,n\}\) and \(\mathcal R\subset\{1,2,\ldots,m\}\), with \(|\mathcal C|=|\mathcal R|=k\).

\STATE Form
\[
    C=A(:,\mathcal C), \qquad R=A(\mathcal R,:), \qquad W=A(\mathcal R,\mathcal C).
\]

\IF{the Frobenius-optimal middle matrix for the fixed \(C,R\) is required}
    \STATE Compute
    \[
    U=C^\dagger A R^\dagger.
    \]
\ELSE
    \STATE Compute the cross factor
    \[
    U=W^\dagger.
    \]
\ENDIF

\RETURN The factors \(C,U,R\), representing \(A\approx CUR\).
\end{algorithmic}
\end{algorithm}

\section{Numerical Examples}
\label{sec:numerical}
The first three examples are MATLAB implementation checks. They were run in MATLAB R2023a on a computer with an Intel Core i7-12700H CPU and 32 GB of RAM. For these tests, the embedded Newton--Schulz iteration is stopped when
\[
    \frac{\|A^\sigma X_kA^\sigma-A^\sigma\|_F}{\|A^\sigma\|_F}<10^{-8},
\]
or after \(500\) iterations. MATLAB's SVD-based \texttt{pinv} provides the reference solution. These timings are based on one matrix realization and are therefore used only for implementation checks.

Example~\ref{subsec:souriau_initialization_experiments} gives the controlled multi-seed study of the polynomial initialization. It uses the same real representation as the theory, includes the initialization cost, and reports variability across independent matrices.

\begin{example}\label{subsec:ex1}
\textbf{(Small dense split-quaternion matrix).}
We generate $A\in\mathbb{Q}_s^{50\times60}$
with independent standard Gaussian entries in each of its four real components. Since the map \(A\mapsto A^\sigma\) is bijective and the distribution is continuous, the real representative $A^\sigma\in\mathbb{R}^{100\times120}$
has full row rank with probability one.

We compare the direct computation
\[
    (A^\sigma)^\dagger=\texttt{pinv}(A^\sigma)
\]
with Newton--Schulz initialized by
\[
    X_0=\frac{0.9}{\|A^\sigma\|_2^2}(A^\sigma)^T.
\]
After convergence, the relative difference from the direct reference is
\[
    \frac{\|X_{\mathrm{NS}}-(A^\sigma)^\dagger\|_F} {\|(A^\sigma)^\dagger\|_F} \approx 2.3\times10^{-12}.
\]

\begin{table}[H]
\centering
\caption{Implementation check for a \(50\times60\) split-quaternion matrix.}
\label{tab:small}
\begin{tabular}{lcc}
\toprule
Method & Time (s) & Iterations \\
\midrule
Direct SVD-based \texttt{pinv} & \(0.42\) & -- \\
Newton--Schulz & \(0.11\) & \(19\) \\
\bottomrule
\end{tabular}
\end{table}

For this realization, the iterative method reaches the direct reference accuracy and is faster in the reported MATLAB implementation. This example is mainly a check of the embedding, inverse map, and stopping criterion.
\end{example}

\begin{example}\label{subsec:ex2}
\textbf{(Medium dense split-quaternion matrix).}
We next generate $A\in\mathbb{Q}_s^{200\times250}$
with the same distribution. Its real representative has size \(400\times500\) and is full row rank with probability one. Table~\ref{tab:medium} compares the direct SVD-based pseudoinverse with the standard Newton--Schulz iteration. The residual-checked extrapolation heuristic from the extrapolation discussion in Section~\ref{Sec:pro} is not used in this comparison.

\begin{table}[H]
\centering
\caption{Implementation check for a \(200\times250\) split-quaternion matrix.}
\label{tab:medium}
\begin{tabular}{lccc}
\toprule
Method & Time (s) & Iterations & Relative error \\
\midrule
Direct SVD-based \texttt{pinv} & \(14.2\) & -- & -- \\
Newton--Schulz & \(3.8\) & \(24\) & \(4.1\times10^{-11}\) \\
\bottomrule
\end{tabular}
\end{table}

The Newton--Schulz computation is about \(3.7\) times faster for this realization while remaining close to the direct reference. The ratio is implementation- and hardware-dependent; it is not an asymptotic complexity claim.
\end{example}

\begin{example}\label{subsec:ex3}
\textbf{(Cross approximation of a low-rank split-quaternion matrix).}
We construct
\[
    A=UV^H, \qquad U\in\mathbb{Q}_s^{300\times20}, \qquad V\in\mathbb{Q}_s^{400\times20},
\]
with random entries. The factorization uses \(20\) split-quaternion latent components. For generic factors,
\[
    \operatorname{rank}_\sigma(A)=40.
\]
We select \(20\) split-quaternion columns and rows by applying a block-aware max-volume procedure to \(A^\sigma\), so that paired real rows and columns are kept together. We then form
\[
    C=A(:,\mathcal C), \qquad R=A(\mathcal R,:), \qquad W=A(\mathcal R,\mathcal C).
\]

We compare the Frobenius-optimal middle factor
\[
    U_\star=C^\dagger A R^\dagger
\]
with the cross factor \(W^\dagger\). For the latter, the \(20\times20\) split-quaternion core corresponds to a \(40\times40\) real matrix. Its pseudoinverse is computed either by a direct SVD or by Newton--Schulz.

\begin{table}[H]
\centering
\caption{Cross approximation of a \(300\times400\) split-quaternion matrix with \(20\) latent split-quaternion components.}
\label{tab:cross}
\begin{tabular}{lcc}
\toprule
Middle factor & Time (s) &
\(\|A-CUR\|_F/\|A\|_F\) \\
\midrule
\(C^\dagger A R^\dagger\) & \(5.1\) & \(1.2\times10^{-14}\) \\
Direct \(W^\dagger\) & \(0.038\) & \(3.4\times10^{-14}\) \\
Newton--Schulz \(W^\dagger\) & \(0.009\) & \(4.1\times10^{-12}\) \\
\bottomrule
\end{tabular}
\end{table}

The selected cross is rank revealing, and both direct middle factors reconstruct the matrix to nearly machine precision. The iterative core pseudoinverse is about \(4.2\) times faster than the direct core pseudoinverse in this implementation, with a reconstruction error of order \(10^{-12}\). Since the core is small, this result should be read as an implementation check rather than as evidence of a different asymptotic cost.
\end{example}

\begin{example}\label{subsec:souriau_initialization_experiments} \textbf{(Polynomial initialization for Newton--Schulz).}
We now test the polynomial initialization introduced in Section~\ref{subsec:souriau-poly-init}. 
The complete Python implementation is available in  the \href{https://colab.research.google.com/drive/1faZlKppRoyZSKIZaMydV04vxqeFIKW4P?usp=sharing} {\texttt{Google Colab notebook}}.

All computations use the real representation fixed in Section~\ref{sec:preliminaries}. Thus, for
\[
    A=A_1+A_2\mathbf{i}+A_3\mathbf{j}+A_4\mathbf{k} \in\mathbb{Q}_s^{m\times n},
\]
we work with
\[
    M=A^\sigma = \begin{bmatrix}
    A_1+A_3 & -A_2+A_4\\
    A_2+A_4 & A_1-A_3
    \end{bmatrix}.
\]
The test matrices have full column rank. We use the iteration
\[
    X_{k+1}=X_k(2I-MX_k)
\]
and stop when
\[
    \frac{\|I-X_kM\|_F}{\sqrt{2n}}<10^{-10},
\]
with at most \(300\) iterations.

The matrices are constructed with prescribed singular values. We generate real matrices \(U\in\mathbb{R}^{m\times n}\) and \(V\in\mathbb{R}^{n\times n}\) with orthonormal columns, and set
\[
    A = \sum_{\ell=1}^{n} \sigma_\ell u_\ell v_\ell^T q_\ell.
\]
The values \(\sigma_\ell\) are geometrically distributed between \(1\) and \(1/\kappa\). The factors \(q_\ell\) are random split-quaternion units of the
form
\[
    \cos(\theta_\ell)+\sin(\theta_\ell)\mathbf{i} \quad\text{or}\quad \cos(\theta_\ell)\mathbf{j}+\sin(\theta_\ell)\mathbf{k}.
\]
Their real representatives are orthogonal. Therefore, the singular values of $M=A^\sigma$
are the prescribed values \(\sigma_\ell\), each repeated twice, and
$\kappa_2(M)=\kappa.$
This construction gives nonzero contributions in all four split-quaternion components while keeping the spectrum under exact control.

Since the largest prescribed singular value is equal to one, we also have
\[
    \|M\|_2=1, \qquad \rho=\|M\|_2^2=1.
\]
Thus, the normalization is known exactly, and no spectral-norm estimation is required for the initializations in these controlled experiments.

We compare four initializations:
\begin{enumerate}
\item the scaled transpose
\[
    X_0=\frac{1}{\rho}M^T, \qquad \rho=\|M\|_2^2;
\]
\item a degree-\(4\) truncated Neumann polynomial;
\item a degree-\(4\) Chebyshev inverse polynomial using the known interval
\[
    [\kappa^{-2},1]
\]
of the normalized Gram matrix;
\item the proposed degree-\(4\) LS--Gram polynomial.
\end{enumerate}
The Chebyshev method is an informed baseline: it uses the exact lower and upper spectral bounds. The LS--Gram construction uses the matrix itself and the upper scaling \(\rho\), but does not require the lower spectral bound.
For the Neumann baseline, we use
\[
p_d^{\mathrm{Neu}}(t)
=
\sum_{\ell=0}^{d}(1-t)^\ell,
\qquad
X_0^{\mathrm{Neu}}
=
\frac{1}{\rho}p_d^{\mathrm{Neu}}(B)M^T.
\]

For each data point, we use five independent matrices and repeat each timing three times. We report medians, with interquartile ranges in the figures. The total time includes both the construction of the initial point and the Newton--Schulz refinement. The reference pseudoinverse is computed with \texttt{numpy.linalg.pinv} using \texttt{rcond}$=10^{-14}$. The reported run used Python~3.12.13, NumPy~2.0.2, and OpenBLAS~0.3.27 with two threads in a single Google Colab session. Since Colab hardware can vary, the timing comparison is interpreted only within this run.

We first study the polynomial degree on split-quaternion matrices of size \(200\times100\), with \(\kappa=10^3\). Table~\ref{tab:lsgram-degree} shows the results. Increasing the degree improves the initial residual and reduces the iteration count. The gain becomes small beyond degree \(4\): degree \(6\) saves one additional iteration, but gives almost the same median total time. We therefore use degree \(4\) in the remaining tests.

\begin{table}[ht!]
\centering
\caption{Degree ablation for split-quaternion matrices of size
\(200\times100\), with \(\kappa=10^3\). The values are medians over five
independent matrices. The total time includes initialization and refinement.}
\label{tab:lsgram-degree}
\begin{tabular}{lcccc}
\hline
Initialization
& \(\|I-X_0M\|_F/\sqrt{2n}\)
& Iter.
& Time (s)
& Relative error \\ \hline
Scaled transpose & \(0.942\) & \(25\) & \(0.115\) & \(7.36\times10^{-14}\) \\
LS--Gram, \(d=0\) & \(0.925\) & \(24\) & \(0.114\) & \(7.27\times10^{-14}\) \\
LS--Gram, \(d=2\) & \(0.858\) & \(21\) & \(0.108\) & \(1.49\times10^{-11}\) \\
LS--Gram, \(d=4\) & \(0.818\) & \(20\) & \(0.105\) & \(7.50\times10^{-14}\) \\
LS--Gram, \(d=6\) & \(0.789\) & \(19\) & \(0.103\) & \(8.20\times10^{-14}\) \\ \hline
\end{tabular}
\end{table}

We next fix \(\kappa=10^3\) and consider the split-quaternion sizes
\[
    100\times50,\qquad 200\times100,\qquad 400\times200,\qquad 600\times300.
\]
The iteration counts are independent of the dimension for this controlled spectrum: the scaled transpose requires \(25\) iterations, the Neumann polynomial \(22\) iterations for the three largest sizes, the Chebyshev polynomial \(19\) iterations, and LS--Gram \(20\) iterations. Thus, LS--Gram reduces the iteration count by \(20\%\) relative to the scaled transpose.

Figure~\ref{fig:lsgram-time-size} reports the complete time. The polynomial setup is visible for the smallest matrix, where all methods have comparable cost. For the three larger sizes, LS--Gram reduces the median total time by about \(13\%\), \(7\%\), and \(7\%\), respectively, compared with the scaled transpose. The Chebyshev method is faster for some intermediate sizes, but it uses the exact spectral interval. At the largest tested size, LS--Gram has the lowest median time of $2.29\ \text{s}$ 

compared with \(2.46\) s for the scaled transpose.

\begin{figure}[ht!]
    \centering
    \includegraphics[width=0.72\textwidth]
    {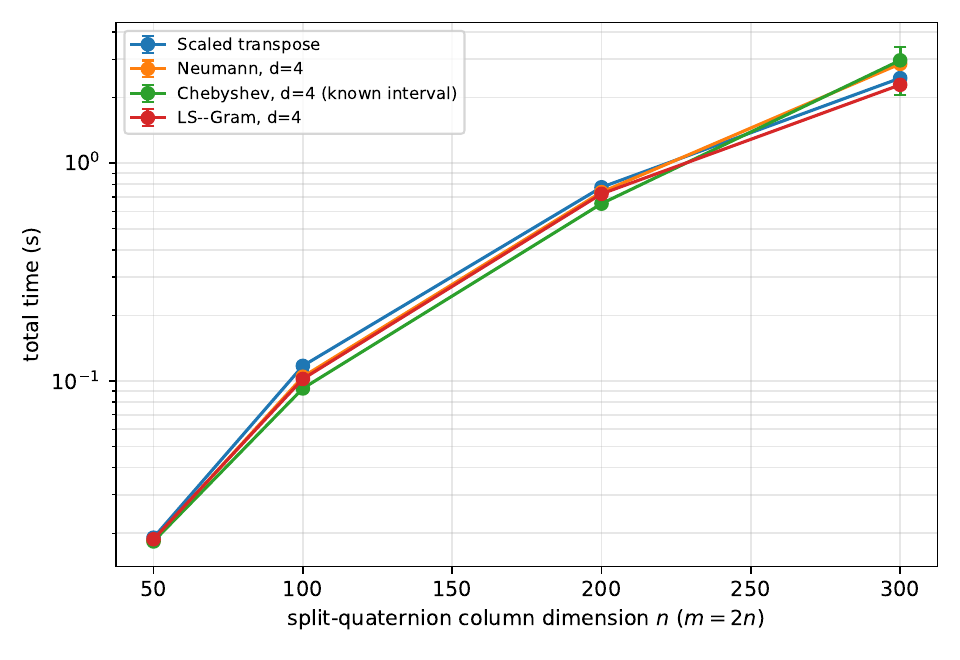}
    \caption{Total time as the matrix size increases, with \(\kappa=10^3\). The split-quaternion matrices have size \(m\times n\) with \(m=2n\). The total time includes initialization and Newton--Schulz refinement. The points and error bars are the median and interquartile range over five matrices.}
    \label{fig:lsgram-time-size}
\end{figure}

Finally, we fix the split-quaternion size to \(300\times150\) and vary
\[
    \kappa\in\{10,10^2,10^3,10^4\}.
\]
Figure~\ref{fig:lsgram-iters-condition} shows that the advantage of the polynomial initializations remains visible as the problem becomes more ill-conditioned. The LS--Gram iteration counts are
\[
    7,\quad 13,\quad 20,\quad 27,
\]
compared with
\[
    12,\quad 18,\quad 25,\quad 31
\]
for the scaled transpose. The informed Chebyshev baseline requires
\[
    6,\quad 13,\quad 19,\quad 26
\]
iterations. Thus, LS--Gram remains within one iteration of Chebyshev over the whole range, without using the lower spectral bound.

\begin{figure}[ht!]
    \centering
    \includegraphics[width=0.72\textwidth]
    {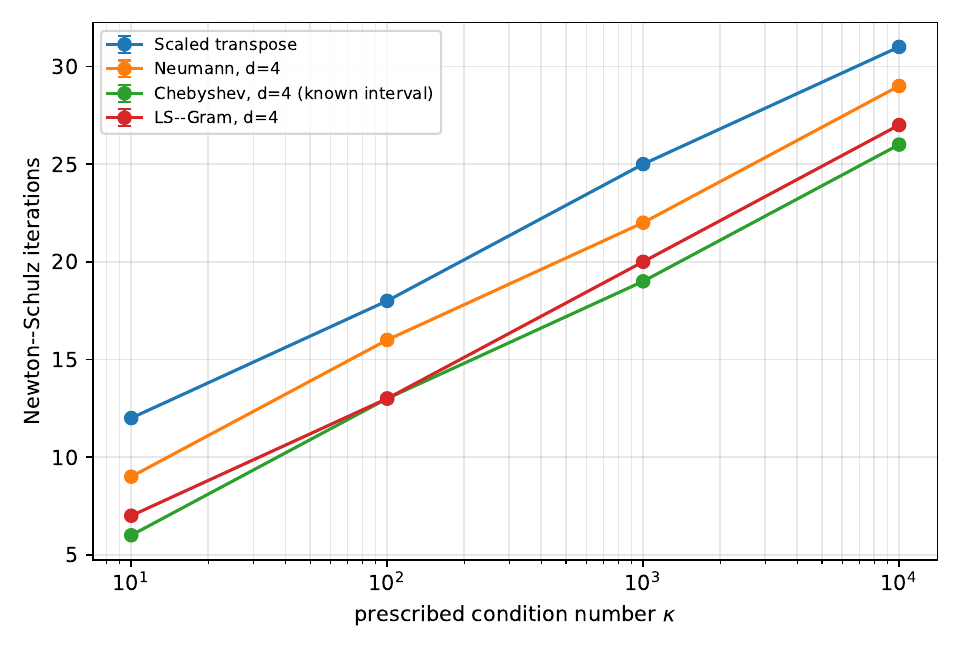}
    \caption{Newton--Schulz iterations as the prescribed condition number increases for split-quaternion matrices of size \(300\times150\). The points and error bars are the median and interquartile range over five matrices.}
    \label{fig:lsgram-iters-condition}
\end{figure}

All LS--Gram initializations passed the spectral acceptance test, and all runs converged. The relative errors with respect to the reference pseudoinverse and the four Penrose residuals were at or below the prescribed numerical accuracy. The notebook also contains tests with clustered and randomly perturbed logarithmic spectra; the same conclusions were observed.

These results support three conclusions. First, the LS--Gram fit gives a consistently better starting point than the scaled transpose. Second, degree \(4\) gives a good compromise between setup cost and iteration reduction. Third, the method is competitive with a Chebyshev polynomial supplied with the exact spectral interval, while requiring less spectral information.
\end{example}

\section{Conclusion and Future Work}
\label{sec:conclusion}

We developed a consistent framework for computing the Moore--Penrose inverse of split-quaternion matrices. The \(2\times 2\) real representation clarifies the role of the (i)-conjugate transpose and gives a direct treatment of rectangular, rank-deficient, and zero-divisor cases. It also allows us to transfer the Newton--Schulz analysis from real matrices to both embedded and native split-quaternion iterations.

We also introduced a low-degree LS--Gram polynomial initialization inspired by Souriau's inverse recursion. The method adapts the initial point to the Gram matrix, while a spectral acceptance test and a safe fallback preserve the convergence guarantee. The numerical experiments show that this initialization improves the initial residual and reduces the number of Newton--Schulz iterations. It also remains competitive with an informed Chebyshev initialization.

Finally, we extended cross and CUR approximation results to split-quaternion matrices. For fixed sampled rows and columns, we characterized a Frobenius-optimal middle factor and gave a rank-revealing condition under which the cheaper cross factor yields an exact reconstruction.

Possible extensions include more stable polynomial bases, adaptive degree selection, matrix-free residual estimates, and larger sparse experiments. The real representation may also be used to define and study the Drazin inverse of square split-quaternion matrices. A complete analysis of this extension is left for future work.

\bibliographystyle{elsarticle-num} 
\bibliography{ref}

\appendix
\section{Auxiliary proofs}
\label{app:auxiliary-proofs}
This appendix collects proofs of standard representation and Moore--Penrose identities used in the main text.

\subsection{Proof of Lemma~\ref{lem:sigma_properties}}
\label{app:proof-sigma-properties}
\begin{proof}
The additivity and real homogeneity follow directly from the block formula. For the product property, write \(A=(a_{ij})\) and \(B=(b_{jk})\). If the two real coordinates associated with each scalar entry are interlaced, the entrywise real representative is the block matrix \([a_{ij}^\sigma]\). The matrix in \eqref{eq:sigma_matrix} is obtained from this entrywise representative by fixed row and column permutations that group the first real coordinates and the second real coordinates. For each \(i,k\),
\[
    ((AB)_{ik})^\sigma =\left(\sum_j a_{ij}b_{jk}\right)^\sigma =\sum_j a_{ij}^\sigma b_{jk}^\sigma,
\]
so the entrywise representative of \(AB\) is the product of the entrywise representatives of \(A\) and \(B\). The intermediate permutation matrices cancel, and therefore \((AB)^\sigma=A^\sigma B^\sigma\), proving \eqref{thir_pro}.

To verify the inverse formula, substitute \eqref{eq:sigma_inverse} into \eqref{eq:sigma_matrix}. For example,
\[
    X_1+X_3 =\frac{Y_{11}+Y_{22}+Y_{11}-Y_{22}}{2} =Y_{11},
\]
and similarly
\[
    -X_2+X_4=Y_{12},\qquad X_2+X_4=Y_{21},\qquad X_1-X_3=Y_{22}.
\]
Thus \(X^\sigma=Y\). The formulas also show uniqueness, so \(\sigma\) is bijective.
\end{proof}

\subsection{Proof of Lemma~\ref{lem:tau_transpose}}
\label{app:proof-tau-transpose}
\begin{proof}
The identities \(\tau^2(q)=q\) and \((\tau(q))^\sigma=(q^\sigma)^T\) follow by direct substitution in \eqref{eq:sigma_scalar}. Using the injectivity of \(\sigma\) and the product property,
\[
    (\tau(pq))^\sigma =((pq)^\sigma)^T =(p^\sigma q^\sigma)^T =(q^\sigma)^T(p^\sigma)^T =(\tau(q)\tau(p))^\sigma.
\]
Therefore \(\tau(pq)=\tau(q)\tau(p)\). The matrix identities follow entrywise, taking into account that transposition reverses the order of matrix multiplication.
\end{proof}

\subsection{Proof of Lemma~\ref{lem:norm_relation}}
\label{app:proof-norm-relation}
\begin{proof}
Expanding the four blocks in \eqref{eq:sigma_matrix} gives
\begin{align*}
    \|A^\sigma\|_F^2 ={}&\|A_1+A_3\|_F^2+\|-A_2+A_4\|_F^2\\
    &+\|A_2+A_4\|_F^2+\|A_1-A_3\|_F^2.
\end{align*}
Using the parallelogram identity twice,
\[
    \|U+V\|_F^2+\|U-V\|_F^2 =2\|U\|_F^2+2\|V\|_F^2,
\]
we obtain
\[
    \|A^\sigma\|_F^2 =2\sum_{\ell=1}^4\|A_\ell\|_F^2 =2\|A\|_F^2.
\]
\end{proof}

\subsection{Proof of Proposition~\ref{prop:mp_basic_properties}}
\label{app:proof-basic-mp}
\begin{proof}
We prove the identities through the real representation. For example,
\[
    \bigl((A^\dagger)^\dagger\bigr)^\sigma =\bigl((A^\dagger)^\sigma\bigr)^\dagger =\bigl((A^\sigma)^\dagger\bigr)^\dagger =A^\sigma.
\]
Injectivity gives \((A^\dagger)^\dagger=A\). Likewise,
\begin{align*}
    \bigl((A^H)^\dagger\bigr)^\sigma
    &=\bigl((A^H)^\sigma\bigr)^\dagger =\bigl((A^\sigma)^T\bigr)^\dagger\\
    &=\bigl((A^\sigma)^\dagger\bigr)^T =\bigl((A^\dagger)^H\bigr)^\sigma,
\end{align*}
which proves the second identity. The scaling identity follows from the corresponding real identity.

If \(U\) and \(V\) are \(H\)-unitary, then \(U^\sigma\) and \(V^\sigma\) are real orthogonal matrices. Therefore,
\begin{align*}
    \bigl((UAV)^\dagger\bigr)^\sigma
    &=\bigl(U^\sigma A^\sigma V^\sigma\bigr)^\dagger\\
    &=(V^\sigma)^T(A^\sigma)^\dagger(U^\sigma)^T\\
    &=\bigl(V^H A^\dagger U^H\bigr)^\sigma.
\end{align*}
Injectivity proves the fourth identity. Finally, the Penrose equations imply
\[
    (AA^\dagger)^H=AA^\dagger, \qquad (A^\dagger A)^H=A^\dagger A,
\]
and
\[
    (AA^\dagger)^2=A(A^\dagger A A^\dagger)=AA^\dagger, \qquad (A^\dagger A)^2=A^\dagger(AA^\dagger A)=A^\dagger A.
\]
\end{proof}

\end{document}